\documentclass[12pt]{amsart}

\usepackage{framed}
\usepackage{braket}
\usepackage{amsmath,amssymb,amsthm}
\usepackage{color}
\usepackage{bm}
\usepackage[all]{xy}
\usepackage{amscd}
\usepackage{graphicx}
\usepackage{ascmac}
\usepackage{BOONDOX-frak, mathrsfs}
\usepackage[top=30truemm,bottom=30truemm,left=25truemm,right=25truemm]{geometry}
\usepackage{mathtools}
\mathtoolsset{showonlyrefs=true}
\usepackage{tikz}
\usetikzlibrary{cd}

\makeatletter
\@addtoreset{equation}{section}

\makeatother

\newcommand{\Z}{\mathbb{Z}}

\newcommand{\Q}{\mathbb{Q}}
\newcommand{\R}{\mathbb{R}}

\newcommand{\F}{\mathbb{F}}

\newcommand{\bA}{\mathbb{A}}

\newcommand{\bF}{\mathbb{F}}
\newcommand{\bG}{\mathbb{G}}

\newcommand{\bN}{\mathbb{N}}

\newcommand{\bZ}{\mathbb{Z}}

\newcommand{\ff}{\mathfrak{f}}
\newcommand{\fg}{\mathfrak{g}}

\newcommand{\fm}{\mathfrak{m}}

\newcommand{\fp}{\mathfrak{p}}

\newcommand{\cD}{\mathcal{D}}

\newcommand{\cF}{\mathcal{F}}

\newcommand{\cK}{\mathcal{K}}

\newcommand{\cO}{\mathcal{O}}

\newcommand{\wtil}[1]{\widetilde{#1}}
\newcommand{\ol}[1]{\overline{#1}}
\newcommand{\parenth}[1]{\left( #1 \right)}

\DeclareMathOperator{\Gal}{Gal}

\DeclareMathOperator{\Cok}{Cok}
\DeclareMathOperator{\Ker}{Ker}

\DeclareMathOperator{\ram}{ram}
\DeclareMathOperator{\fin}{fin}

\DeclareMathOperator{\id}{id}
\DeclareMathOperator{\Cl}{Cl}
\DeclareMathOperator{\ord}{ord}

\DeclareMathOperator{\rank}{rank}

\DeclareMathOperator{\cha}{char}

\DeclareMathOperator{\Tr}{Tr}

\usepackage[OT2,T1]{fontenc}
\DeclareSymbolFont{cyrletters}{OT2}{wncyr}{m}{n}
\DeclareMathSymbol{\Sha}{\mathalpha}{cyrletters}{"58}

\let\oldenumerate\enumerate
\renewcommand{\enumerate}{
   \oldenumerate
   \setlength{\itemsep}{1pt}
   \setlength{\parskip}{0pt}
   \setlength{\parsep}{0pt}
}
\let\olditemize\itemize
\renewcommand{\itemize}{
   \olditemize
   \setlength{\itemsep}{1pt}
   \setlength{\parskip}{0pt}
   \setlength{\parsep}{0pt}
}

\theoremstyle{plain}
\newtheorem{thm}{Theorem}[section]
\newtheorem{lem}[thm]{Lemma}
\newtheorem{conj}[thm]{Conjecture}
\newtheorem{prop}[thm]{Proposition}
\newtheorem{cor}[thm]{Corollary}

\theoremstyle{definition}
\newtheorem{defn}[thm]{Definition}

\newtheorem{rem}[thm]{Remark}
\newtheorem{eg}[thm]{Example}

\usepackage{booktabs}
\usepackage{multirow}

\DeclareMathOperator{\length}{length}

\DeclareMathOperator{\Lie}{Lie}

\DeclareMathOperator{\sep}{sep}
\DeclareMathOperator{\tors}{tors}
\DeclareMathOperator{\constant}{(constant)}

\title
[Taelman class groups]
{Iwasawa-type asymptotic formula for Taelman class groups of Drinfeld modules}
\author{Takenori Kataoka}
\address{Department of Mathematics, Faculty of Science Division II, Tokyo University of Science.
1-3 Kagurazaka, Shinjuku-ku, Tokyo 162-8601, Japan}
\email{tkataoka@rs.tus.ac.jp}

\author{Yoshiaki Okumura}
\address{Department of Architecture, Faculty of Science and Engineering,
Toyo University.
2100, Kujirai, Kawagoe, 
Saitama 350-8585, Japan}
\email{okumura165@toyo.jp}


\date{\today}
\begin{document}

\maketitle

%%%%%%%%%%%%%%%%%%%%%
\begin{abstract}
The Taelman class groups associated to Drinfeld modules over function fields serve as an analogue of ideal class groups of number fields.
In this paper, we establish an analogue of Iwasawa's asymptotic formula for $\Z_p$-extensions in the context of the Taelman class groups.
For this purpose, we naturally introduce and investigate Taelman class groups with moduli.
\end{abstract}
%%%%%%%%%%%%%%%%%%%%%

%%%%%%%%%%%%%%%%%%%%%%
\section{Introduction}\label{sec:intro}
%%%%%%%%%%%%%%%%%%%%%%

%%%%%%%%%%%%%%%%%%%%%%
\subsection{Iwasawa's asymptotic formula for ideal class groups}\label{ss:intro_1}
%%%%%%%%%%%%%%%%%%%%%%

We begin with a quick review of Iwasawa's asymptotic formula that describes the behavior of the ideal class groups of number fields.
For a number field $K$, let $\Cl(K)$ be the ideal class group, which is known to be a finite abelian group.

Let $p$ be a fixed prime number.
Let $\cK/K$ be a $\Z_p$-extension of number fields, which can be
regarded as
a family of number fields $\{K_n\}_{n \geq 0}$ with
\[
K = K_0 \subset K_1 \subset K_2 \subset \cdots \subset \cK = \bigcup_{n \geq 0} K_n
\]
such that $K_n/K$ is a cyclic extension of degree $p^n$ for any $n \geq 0$.
The intermediate field $K_n$ is called the $n$-th layer of $\cK/K$.
Then Iwasawa's asymptotic formula \cite[Theorem 11]{Iwa59} (see also Washington \cite[Theorem 13.13]{Was97}) states that there exist integers $\lambda \geq 0, \mu \geq 0, \nu$ such that
\[
\length_{\Z_p}(\Z_p \otimes_{\Z} \Cl(K_n)) = \lambda n + \mu p^n + \nu
\]
holds for $n \gg 0$.
Here and henceforth, we write $\length_R(-)$ for the length of modules over a commutative ring $R$.
The left hand side is often written as the $p$-adic valuation of the order of $\Cl(K_n)$, but this formalism using the length is more suitable in this paper.

%%%%%%%%%%%%%%%%%%%%%%
\subsection{Taelman class groups}\label{ss:intro_3}
%%%%%%%%%%%%%%%%%%%%%%

The main purpose of this paper is to establish an analogue of Iwasawa's asymptotic formula for Taelman class groups associated to Drinfeld modules over function fields.
See \S \ref{sec:NT} for comparison with the number field case.

We fix notation to be used throughout this paper.
Let $q$ be a power of a prime number $p$.
Let $Q$ be a global function field with constant field $\bF_q$.
We fix a place $\infty$ of $Q$ and denote by $A\subset Q$ the ring of elements regular outside $\infty$.
A typical example is $Q = \F_q(t)$ with $\infty$ corresponding to $1/t$, in which case we have $A = \F_q[t]$.

Let $K$ be a global function field over $\F_q$.
We assume that $K$ is a $Q$-algebra, that is, $K$ is equipped with a (necessarily injective) $\bF_q$-algebra homomorphism $\gamma: Q \hookrightarrow K$.
We call such $K$ a {\it global $Q$-field} for convenience.
We define the integer ring $\cO_K$ of $K$ as the integral closure of $\gamma(A)$ in $K$.
The maximal ideals of $\cO_K$ are called finite places of $K$, and the other places are called infinite places.
For a finite place $v$ of $K$, its restriction to $A$ via $\gamma$ is a prime (i.e., a maximal ideal) $\fp$ of $A$, and in this case $v$ is said to be a $\fp$-adic place.

Let $E$ be a Drinfeld $A$-module over $\cO_K$ (Definition \ref{defn:Drinf}).
In \cite{Tae10}, Taelman introduced a finite module $H(E/\cO_K)$ over $A$ (not over $\Z$), which we call the {\it Taelman class group}. It is shown that
the Taelman class groups serve
as an analogue of the ideal class groups of number fields.
As a keystone result, we have an analogue of the analytic class number formula \cite[Theorem 1]{Tae12}, which was subsequently generalized in various directions by several authors (e.g., \cite{AT15}, \cite{FGHP22}).
Angl\`{e}s--Taelman \cite{AT15} studied $H(C/\cO_K)$
for the Carlitz module $C$ and Carlitz cyclotomic extensions $K/Q$,
and showed results analogous to the ideal class groups $\Cl(K)$ for cyclotomic extensions $K/\Q$.

More recently, Higgins \cite{Hig21} studied the Taelman class groups from Iwasawa theoretic perspective.
In \cite{Hig21}, the inverse limit of the Taelman class groups is studied as a module over the Iwasawa algebra (see \S \ref{ss:prf_outline} below), and an Iwasawa main conjecture is formulated.
However, an Iwasawa-type asymptotic formula is not discussed.

%%%%%%%%%%%%%%%%%%%%%%
\subsection{Main theorem}\label{ss:intro_4}
%%%%%%%%%%%%%%%%%%%%%%

Let $\cK/K$ be a $\Z_p$-extension of global $Q$-fields in which only finitely many places of $K$ are ramified (see Example \ref{eg:Z_p-ext}).
We write $K_n$ for its $n$-th layer.
Let $E$ be a Drinfeld $A$-module over $\cO_K$.

The Taelman class group, being a finite module over the Dedekind domain $A$, decomposes into its $\fp$-parts as
\[
H(E/\cO_{K_n}) \simeq \bigoplus_{\fp} (A_{\fp} \otimes_A H(E/\cO_{K_n})),
\]
where $\fp$ runs over the primes of $A$ and $A_{\fp}$ denotes the completion.
This is analogous to the decomposition of the ideal class group $\Cl(K_n) \simeq \bigoplus_p (\Z_p \otimes_{\Z} \Cl(K_n))$, where $p$ runs over the prime numbers.

The main result of this paper, Theorem \ref{thm:ICNF1}, in the case $S = \emptyset$ reads as follows.

\begin{thm}\label{thm:ICNF}
Let $\fp$ be a prime of $A$.
Suppose that $\cK/K$ is unramified at all $\fp$-adic places.
Then there are integers $\mu_{\fp} \ge 0$ and $\nu_{\fp}$ such that
\[
\length_{A_{\fp}}(A_{\fp} \otimes_A H(E/\cO_{K_n})) 
= \mu_{\fp} p^n + \nu_{\fp}
\]
holds for $n \gg 0$.
\end{thm}

Indeed, we have $\mu_{\fp} = \mu(H(E/\cO_{\cK})_{\fp})$, using Definitions \ref{defn:Iw_mod} and \ref{defn:str_thm}.
In contrast to the number field case, the absence of the $\lambda$-invariant seems to be notable.
This results from a purely algebraic reason (\S \ref{sec:alg}).

By taking the sum with respect to the primes $\fp$, we will easily deduce the following corollary.

\begin{cor}\label{cor:ICNF_ur}
Suppose that $\cK/K$ is unramified at all finite places.
Then there are integers $\mu \ge 0$ and $\nu$ such that
\[
\length_A(H(E/\cO_{K_n})) 
= \mu p^n + \nu
\]
holds for $n \gg 0$.
\end{cor}

The unramified assumption sounds restrictive, though it was also assumed in \cite{Hig21}.
In Theorem \ref{thm:ICNF1}, we will remove the unramified assumption, but instead deal with $S$-modified Taelman class groups with $S$ a finite set of places containing the ramified $\fp$-adic places.

%%%%%%%%%%%%%%%%%%%%%%
\subsection{Strategy of the proof}\label{ss:prf_outline}
%%%%%%%%%%%%%%%%%%%%%%

One of the main ingredients for the proof is a module-theoretic version of Iwasawa's asymptotic formula in positive characteristic, which we establish in \S \ref{sec:alg}.

For a while, we discuss Theorem \ref{thm:ICNF} when $\cK/K$ is unramified at all finite places.
In the same way as Higgins \cite[Chapter 4]{Hig21}, we study the projective limit
\[
H(E/\cO_{\cK})_{\fp} = \varprojlim_n (A_{\fp} \otimes_A H(E/\cO_{K_n}))
\]
as a module over the completed group ring $A_{\fp}[[\Gal(\cK/K)]]$.
Thanks to the unramified condition, we have a Galois descent property
\[
(H(E/\cO_{\cK})_{\fp})_{\Gal(\cK/K_n)} \simeq A_{\fp} \otimes_A H(E/\cO_{K_n})
\]
for each $n \geq 0$.
Then Theorem \ref{thm:ICNF} follows by applying the result in \S \ref{sec:alg}.

However, if $\cK/K$ is ramified at a ($\fp$-adic or non-$\fp$-adic) finite place, then the Galois descent property fails.
This kind of issue already appears in previous work such as Ferrara--Green--Higgins--Popescu \cite{FGHP22}.
The solution in \cite{FGHP22} is to introduce auxiliary {\it taming modules}.
However, the choice of taming module is not unique, and we would like to avoid such an ambiguity.

For this reason, we take another solution, that is, we introduce $S$-modified Taelman class groups for a finite set $S$ of places.
We will show that the $S$-modified Taelman class groups satisfy the Galois descent property as long as $S$ contains the ramified places.
Finally, by comparing the $S$-modified Taelman class groups with the original ones, we obtain Theorem \ref{thm:ICNF} and its generalization, Theorem \ref{thm:ICNF1}.

%%%%%%%%%%%%%%%%%%%%%%
\subsection{Structure of this paper}\label{ss:structure}
%%%%%%%%%%%%%%%%%%%%%%

In \S \ref{sec:Drinf}, we review basic notations concerning Drinfeld modules.
In \S \ref{sec:clgp}, we introduce the unit groups and class groups \'{a} la Taelman and their $S$-modifications.
In \S \ref{sec:app}, we prove the main theorems.
In \S \ref{sec:rank}, we give remarks on the sizes of $S$-modified Taelman class groups and their Iwasawa-theoretic versions.

In \S \ref{sec:alg}, we show the module-theoretic version of Iwasawa's asymptotic formula.
In \S \ref{sec:NT}, we review the ideal class groups to facilitate comparison with the main theme of this paper.

%%%%%%%%%%%%%%%%%%%%%%
\section{Drinfeld modules}\label{sec:Drinf}
%%%%%%%%%%%%%%%%%%%%%%

In this section, we review the basics of Drinfeld modules, introduced by Drinfeld \cite{Dri74} under the name of {\it elliptic modules}.
We mainly refer the details to the books of Goss \cite{Gos96} and Papikian \cite{Pap23}.
Note that in the main text of \cite{Pap23} we assume $A = \F_q[t]$, while the general case is discussed in its Appendix A.

Let $A \subset Q$ be as in \S\ref{ss:intro_3}.
We write $Q_{\infty}$ and $Q_{\fp}$ for the completions at $\infty$ and a prime $\fp$ of $A$, respectively.
For a global $Q$-field $K$, we use the following notation.
For a place $v$ of $K$, we write $v \mid \infty$ (resp.~$v \mid \fp$) if $v$ is an infinite place (resp.~a $\fp$-adic place).
Let $S_{\infty} = S_{\infty}(K)$ (resp.~$S_{\fp} = S_{\fp}(K)$) denote the set of infinite places (resp.~$\fp$-adic places) of $K$.
For each place $v$, let $K_v$ denote the completion of $K$ at $v$.
We write $K_{\infty} = Q_{\infty} \otimes_Q K = \prod_{v \mid \infty} K_v$.
When $v$ is a finite place, we also have the integer ring $\cO_{K, v} \subset K_v$, the maximal ideal $\fm_{K, v} \subset \cO_{K, v}$, and the residue field $\kappa_v(K) = \cO_{K, v}/\fm_{K, v}$.

%%%%%%%%%%%%%%%%%%%%%%
\subsection{Drinfeld modules}\label{ss:Drinf}
%%%%%%%%%%%%%%%%%%%%%%

For a commutative $\F_q$-algebra $R$, we have the $q$-th Frobenius map
\[
\tau: R \to R,
\qquad
\tau(a) = a^q.
\]
We write $R\{\tau\}$ for the non-commutative polynomial ring.
By definition, $R\{\tau\}$ consists of polynomials of $\tau$ with coefficients in $R$, and the multiplication is defined so that $(a \tau^i) (b \tau^j) = a \tau^i(b) \tau^{i+j}$ for $a, b \in R$ (\cite[Definition 3.1.8]{Pap23}).
We have an $R$-algebra homomorphism 
\[
\partial: R\{\tau\} \to R,  
\qquad
a_0 + a_1 \tau + \dots + a_n \tau^n \mapsto a_0,
\]
which is called the derivative (\cite[Definition 3.1.12]{Pap23}).

Let $K$ be a global $Q$-field with an $\F_q$-homomorphism $\gamma: Q \hookrightarrow K$.

\begin{defn}[{\cite[Definition 3.2.2, Definition A.2]{Pap23}}]\label{defn:Drinf}
A Drinfeld $A$-module $E$ over $\cO_K$ is an $\F_q$-algebra homomorphism
\[
\phi_E: A \to \cO_K\{\tau\}
\]
such that the composite map $\partial \circ \phi_E: A \to \cO_K$ coincides with $\gamma$.
Note that we do not assume the usual extra condition that $\phi_E(A) \not \subset \cO_K$; see Remark \ref{rem:triv}.
\end{defn}

In the rest of this section, let $E$ be a Drinfeld $A$-module over $\cO_K$.
The scalar restriction via $\phi_E$ gives us a functor
\[
E: \text{(left $\cO_K \{\tau\}$-modules)} \to \text{($A$-modules)}.
\]
Concretely, for each $\cO_K \{\tau\}$-module $M$, we set $E(M) := M$ on which $A$ acts as
\[
a \cdot x = \phi_E(a) x
\]
for $a \in A$ and $x \in M$.
Note that $M$ and $E(M)$ are identical as $\F_q$-modules, so the functor is exact.
Since $E$ is a functor, if $M$ is equipped with an action of a group $G$, then $E(M)$ also has an action of $G$.
Moreover, by the construction, the $\F_q[G]$-module structures of $M$ and $E(M)$ are identical.
In particular, $M$ is $G$-cohomologically trivial ($G$-c.t., for short) if and only if so is $E(M)$.

We also have a natural functor
\[
\text{(commutative $\cO_K$-algebras)} \to \text{(left $\cO_K \{\tau\}$-modules)}
\]
by considering each commutative $\cO_K$-algebra $R$ as a left $\cO_K\{\tau\}$-module by using the $q$-th Frobenius map $\tau: R \to R$.
The resulting composite functor 
\[
\text{(commutative $\cO_K$-algebras)} \to \text{(left $\cO_K \{\tau\}$-modules)} \overset{E}{\to} \text{($A$-modules)}
\]
is also denoted by $E$.

It is known (\cite[Definition A.6]{Pap23}) that there is a non-negative integer $r$ such that for any $a\in A$, 
\[
\deg_\tau\phi_E(a)=r\cdot \dim_{\bF_q}(A/aA),
\] 
where $\deg_\tau \phi_E(a)$ is the degree as a polynomial in $\tau$.
We call $r$ the {\it rank} of $E$.

\begin{eg}\label{eg:Carlitz}
Suppose that $Q=\bF_q(t)$ and $\infty$ is the place corresponding to $1/t$, so $A=\bF_q[t]$.
We also write $\theta = \gamma(t) \in K$.
We define the Carlitz module $C$, which is a Drinfeld $A$-module over $\cO_K$ of rank one, by
\[
\phi_C: A = \F_q[t] \to \cO_K\{\tau\},
\quad
\phi_C(t) = \theta + \tau.
\]
For a commutative $\cO_K$-algebra $R$, the $A$-module structure on $C(R)$ is given by $t \cdot a = \theta a + a^q$.
\end{eg}

%%%%%%%%%%%%%%%%%%%%%%
\subsection{Exponential maps}\label{ss:lattice}
%%%%%%%%%%%%%%%%%%%%%%

We write $\phi_{\Lie_E}$ for the homomorphism 
\[
\phi_{\Lie_E} := \partial \circ \phi_E = \gamma: A \to \cO_K.
\]
This defines a functor
\[
\Lie_E: \text{($\cO_K$-modules)} \to \text{($A$-modules)},
\]
which is simply the scalar restriction from $\cO_K$ to $A$ via $\gamma$.

It is known (\cite[Proposition 4.6.7]{Gos96}) that there is a unique power series, called the exponential map,
\[
\exp_E(X) = X + \sum_{i = 1}^{\infty} e_i X^{q^i} \in K[[X]]
\]
such that 
\[
\phi_E(a)(\exp_E(X)) = \exp_E(\gamma(a) X)
\]
for all $a \in A$.

Let $v$ be an infinite place of $K$.
Then it is known (\cite[Proof of Theorem 4.6.9]{Gos96}) that $\exp_E$ is convergent on $K_v$, giving an $A$-homomorphism
\[
\exp_E: \Lie_E(K_v) \to E(K_v).
\]
We write $\Lambda_E(K_v)$ for the kernel of this map.
By taking the direct sum for $v \mid \infty$, we also have an $A$-homomorphism
\[
\exp_E: \Lie_E(K_{\infty}) \to E(K_{\infty}),
\]
whose kernel is $\Lambda_E(K_{\infty}) := \bigoplus_{v \mid \infty} \Lambda_E(K_v)$.

For each $v$, it is known that $\Lambda_E(K_v)$ is a finitely generated $A$-module of rank $\leq r$, where $r$ denotes the rank of $E$.
Indeed, the kernel of $\exp_E: \Lie_E(K_v^{\sep}) \to E(K_v^{\sep})$ is known to be finitely generated of rank $r$, where $K_v^{\sep}$ denotes the separable closure of $K_v$ (see \cite[Proof of Theorem 4.6.9]{Gos96}), so we have $\rank_A \Lambda_E(K_v) = r$ if $K_v$ is large enough.
We set
\[
r_E(K) = \rank_A \Lambda_E(K_{\infty}),
\]
so we have $r_E(K) \leq r \cdot \# S_{\infty}(K)$ and the equality holds if $K$ is large enough.
This invariant $r_E(K)$ plays the role of the number of complex places, usually denoted by $r_2(K)$ for a number field $K$, in the original Iwasawa theory; see \S \ref{ss:HU_rank} and \S \ref{sec:NT}.

\begin{eg}\label{eg:Carlitz_2}
Let us consider the Carlitz module $C$ as in Example \ref{eg:Carlitz}.
If we set $D_0=1$ and $D_i=(\theta^{q^i}-\theta)D_{i-1}^q$ for $i\geq 1$, 
then it is known (\cite[Proposition 5.4.1]{Pap23}) that the exponential map of $C$ is
\[
\exp_C(X)=X+\sum_{i=1}^\infty \frac{X^{q^i}}{D_i}\in \bF_q(\theta)[[X]] \subset K[[X]]
\]
and the kernel of $\exp_C\colon \Lie_C(\F_q((1/\theta))^{\sep}) \to C(\F_q((1/\theta))^{\sep})$ is a rank-one free $A$-module generated by 
\[
\pi_C= (-\theta)^{\frac{1}{q-1}} \theta \prod_{i=1}^\infty \parenth{1-\frac{1}{\theta^{q^i-1}}}^{-1},
\]
where $(-\theta)^{\frac{1}{q-1}}$ is a $(q-1)$-th root of $-\theta$ (see \cite[Corollary 5.4.9]{Pap23}).
Thus,
we have
\[
r_C(K)=\#\{v \in S_\infty(K) \mid (-\theta)^{\frac{1}{q-1}} \in K_v\}.
\] 
\end{eg}

%%%%%%%%%%%%%%%%%%%%%%
\subsection{Local rational points}\label{ss:local}
%%%%%%%%%%%%%%%%%%%%%%

For a finite place $v$ of $K$, we have $A$-modules $E(\fm_{K, v}) \subset E(\cO_{K, v}) \subset E(K_v)$.
The quotient module $E(\cO_{K, v})/E(\fm_{K, v}) \simeq E(\kappa_v(K))$ is finite, and $E(\fm_{K, v})$ has the following property.

\begin{prop}\label{prop:loc_str}
Let $\fp$ be a prime of $A$ and $v$ a $\fp$-adic place of $K$.
Then $E(\fm_{K, v})$ is naturally equipped with a finitely generated $A_{\fp}$-module structure of rank $[K_v: Q_{\fp}]$.
Therefore, $\bigoplus_{v \mid \fp} E(\fm_{K, v})$ is a finitely generated $A_{\fp}$-module of rank $[K: Q]$.
\end{prop}

\begin{proof}
The $A_{\fp}$-module structure on $E(\fm_{K, v})$
comes from the formal Drinfeld $A_\fp$-module attached to $E$
as explained in \cite[Text after Definition 6.5.1]{Pap23} and Rosen \cite[p.247, Text before Definition]{Ros03}.
Moreover, 
in \cite[Proposition 2.3]{Ros03},
it is shown that the logarithm map gives an $A_{\fp}$-isomorphism $E(\fm_{K, v}^s) \to \fm_{K, v}^s$ for a large integer $s \geq 1$.
Therefore, we have 
\begin{align*}
\rank_{A_{\fp}}(E(\fm_{K, v}))
& = \rank_{A_{\fp}}(E(\fm_{K, v}^s))
= \rank_{A_{\fp}}(\fm_{K, v}^s)\\
& = \rank_{A_{\fp}}(\cO_{K, v})
= [K_v: Q_{\fp}]
\end{align*}
as claimed.
The final assertion follows from $\sum_{v \mid \fp} [K_v: Q_{\fp}] = [K: Q]$.
\end{proof}

%%%%%%%%%%%%%%%%%%%%%%
\section{Taelman class groups}\label{sec:clgp}
%%%%%%%%%%%%%%%%%%%%%%

In this section, after reviewing the unit groups and class groups \'{a} la Taelman, we introduce its generalizations by considering moduli.
See \S \ref{sec:NT} for analogues in the number field setting.
Then we observe the behavior of the class groups with respect to Galois descent, which plays the key role in the proof of the main theorems.

Let $K$ be a global $Q$-field and $E$ a Drinfeld $A$-module over $\cO_K$.

%%%%%%%%%%%%%%%%%%%%%%
\subsection{The original unit groups and class groups}\label{ss:clgp}
%%%%%%%%%%%%%%%%%%%%%%

We first review the original unit group and class group (without moduli), introduced by Taelman \cite{Tae10} (see also \cite{ANT17}).

\begin{defn}\label{defn:clgp_ori}
We define the Taelman class group $H(E/\cO_K)$ as the quotient $A$-module
\[
H(E/\cO_K)
:= \frac{E(K_{\infty})}{E(\cO_K) + \exp_E(K_{\infty})}
\]
(\cite[Definition 4]{Tae10}; one may refer to \cite[Theorem 7.6.17(1)]{Pap23}).
Strictly speaking, the term $\exp_E(K_{\infty})$ in the denominator should be written as $\exp_E(\Lie_E(K_{\infty}))$, but this is unnecessarily complicated, so we adopt this simpler notation.

We have two kinds of unit groups, denoted by $U(E/\cO_K)$ and $U'(E/\cO_K)$, the latter being auxiliary in this paper.
(Even the former is unnecessary for the proof of the main results themselves, but is needed to fully study the class groups in \S \ref{ss:HU_rank}.)
We define $U(E/\cO_K)$ as the kernel of the natural $A$-homomorphism
\[
E(\cO_K) \hookrightarrow E(K_{\infty}) \twoheadrightarrow \frac{E(K_{\infty})}{\exp_E(K_{\infty})},
\]
while we define $U'(E/\cO_K)$ as the kernel of the $A$-homomorphism
\[
\Lie_E(K_{\infty}) \overset{\exp_E}{\to} E(K_{\infty}) \twoheadrightarrow \frac{E(K_{\infty})}{E(\cO_K)}.
\]
The module $U'(E/\cO_K)$ appears in \cite[Theorem 1]{Tae10}, while $U(E/\cO_K)$ appears in \cite[Proposition 1]{Tae10} when $E$ is the Carlitz module.
One may refer to \cite[The beginning of \S 7.6.3]{Pap23} for $U(E/\cO_K)$ in general.
\end{defn}

By definition, we have an exact sequence
\[
0 \to U(E/\cO_K)
\to E(\cO_K) 
\to \frac{E(K_{\infty})}{\exp_E(K_{\infty})}
\to H(E/\cO_K)
\to 0
\]
of $A$-modules.

\begin{rem}\label{rem:triv}
In the definition of Drinfeld modules, we do not assume the standard extra condition that $\phi_E(A) \not \subset \cO_K$.
When $\phi_E(A) \subset \cO_K$, the exponential map $\exp_E: \Lie_E(K_v) \to E(K_v)$ is simply the identity map, so we have $H(E/\cO_K) = 0$, which is uninteresting.
\end{rem}

The relation between $U(E/\cO_K)$ and $U'(E/\cO_K)$ is described as follows.

\begin{lem}\label{lem:rel_UU'}
We have an exact sequence
\[
0 \to \Lambda_E(K_{\infty}) \to U'(E/\cO_K) \overset{\exp_E}{\to} U(E/\cO_K) \to 0
\]
of $A$-modules.
\end{lem}

\begin{proof}
We obtain this lemma by applying the snake lemma to the diagram
\[
\xymatrix{
	& & \Lie_E(K_{\infty}) \ar@{=}[r] \ar[d]^{\exp_E}
	& \Lie_E(K_{\infty}) \ar[d]^{\exp_E}
	& \\
	0 \ar[r]
	& E(\cO_K) \ar[r]
	& E(K_{\infty}) \ar[r]
	& \dfrac{E(K_{\infty})}{E(\cO_K)} \ar[r]
	& 0.
}
\]
\end{proof}

\begin{prop}\label{prop:UH_rank}
The following hold.
\begin{itemize}
\item[(1)]
The $A$-module $H(E/\cO_K)$ is finite.
\item[(2)]
The $A$-module $U(E/\cO_K)$ is finitely generated of rank $[K: Q] - r_E(K)$.
\end{itemize}
\end{prop}

\begin{proof}
Take a non-constant element $t\in A\setminus \F_q$.
Restricting $E$ to the subring $\F_q[t]$ of $A$, we can regard $E$ as a Drinfeld $\F_q[t]$-module over $\cO_K$.
Then its exponential map coincides with the original one by uniqueness.
Hence its Taelman class group and unit groups are also the same as $H(E/\cO_K), U'(E/\cO_K)$, and $U(E/\cO_K)$ viewed as $\F_q[t]$-modules, respectively.

(1) By this argument, we may assume that $A=\F_q[t]$.
Then the claim is proved by Taelman \cite[Theorem 1]{Tae10}, which is reproduced in \cite[Theorem 7.6.17(1)]{Pap23}.

(2)
When $E$ is the Carlitz module, this is proved in \cite[Proposition 1]{Tae10}.
The general case can be shown in the same way, as follows.
For $E$ viewed as a Drinfeld $\F_q[t]$-module, 
Taelman \cite[Theorem 1]{Tae10} proved that $U'(E/\cO_K)$ is an $\F_q[t]$-lattice of $\Lie_E(K_{\infty})$ (i.e., a discrete and co-compact $\F_q[t]$-submodule).
Therefore, it is also an $A$-lattice of $\Lie_E(K_{\infty})$, so it is a projective $A$-module of rank $[K: Q]$ (one may refer to \cite[Theorem 7.6.17(2), Corollary 7.6.18]{Pap23}).
Then we obtain the claim by Lemma \ref{lem:rel_UU'}, noting that the $A$-rank of $\Lambda_E(K_{\infty})$ is $r_E(K)$ by definition.
\end{proof}

%%%%%%%%%%%%%%%%%%%%%%
\subsection{Modifications with moduli}\label{ss:clgp_ram}
%%%%%%%%%%%%%%%%%%%%%%

We introduce the concept of unit groups and class groups with moduli.

\begin{defn}\label{defn:clgp_moduli}
Let $\ff$ be a nonzero ideal of $\cO_K$.
We regard $\ff$ as a modulus $\prod_{v} v^{\ord_v(\ff)}$ of $K$, where $v$ runs over the finite places of $K$.
Then we define the unit group $U_{\ff}(E/\cO_K)$ and class group $H_{\ff}(E/\cO_K)$ as, respectively, the kernel and cokernel of the natural $A$-homomorphism
\[
E(K) \to \frac{E(K_{\infty})}{\exp_E(K_{\infty})} \oplus \bigoplus_v \frac{E(K_v)}{E(\ff \cO_{K, v})}
\]
induced by the diagonal map.
Therefore, we have an exact sequence
\begin{equation}\label{eq:fund_ex_1}
0 \to U_{\ff}(E/\cO_K)
\to E(K) 
\to \frac{E(K_{\infty})}{\exp_E(K_{\infty})} \oplus \bigoplus_v \frac{E(K_v)}{E(\ff \cO_{K, v})}
\to H_{\ff}(E/\cO_K)
\to 0
\end{equation}
of $A$-modules.
\end{defn}

\begin{rem}\label{rem:clgp_var}
We have another description of $U_{\ff}(E/\cO_K)$ and $H_{\ff}(E/\cO_K)$.
Let $S$ be a finite set of finite places of $K$ that contains the support of $\ff$ (i.e., $\ord_v(\ff) = 0$ unless $v \in S$).
Let $\cO_{K, S}$ be the ring of $S$-integers of $K$.
Then $U_{\ff}(E/\cO_K)$ and $H_{\ff}(E/\cO_K)$ can be defined as the kernel and cokernel of the homomorphism
\[
E(\cO_{K, S}) \to \frac{E(K_{\infty})}{\exp_E(K_{\infty})} \oplus \bigoplus_{v \in S} \frac{E(K_v)}{E(\ff \cO_{K, v})}.
\]
This is because the natural homomorphism
\[
\frac{K}{\cO_{K, S}} \to \bigoplus_{v \not \in S} \frac{K_v}{\cO_{K, v}}
\]
is an isomorphism.
In particular, when $\ff = (1)$ is the trivial modulus, we may take $S = \emptyset$, so $U_{(1)}(E/\cO_K) = U(E/\cO_K)$ and $H_{(1)}(E/\cO_K) = H(E/\cO_K)$.
\end{rem}

Let us study the effect of changing the moduli.

\begin{prop}\label{prop:compar1}
For moduli $\ff, \fg$ of $K$ with $\fg \mid \ff$ (i.e., $\fg$ is a divisor of $\ff$), we have an exact sequence
\[
0 \to U_{\ff}(E/\cO_K) 
\to U_{\fg}(E/\cO_K)
\to \bigoplus_v \frac{E(\fg \cO_{K, v})}{E(\ff \cO_{K, v})}
\to H_{\ff}(E/\cO_K)
\to H_{\fg}(E/\cO_K)
\to 0.
\]
\end{prop}

\begin{proof}
Consider the natural commutative diagram
\[
\xymatrix{
	0 \ar[r]
	& U_{\ff}(E/\cO_K) \ar[r] \ar[d]
	& E(K) \ar[r] \ar@{=}[d]
	& \frac{E(K_{\infty})}{\exp_E(K_{\infty})} \oplus \bigoplus_v \frac{E(K_v)}{E(\ff \cO_{K, v})} \ar[r] \ar[d]^{(*)}
	& H_{\ff}(E/\cO_K) \ar[r] \ar[d]
	& 0\\
	0 \ar[r]
	& U_{\fg}(E/\cO_K) \ar[r]
	& E(K) \ar[r]
	& \frac{E(K_{\infty})}{\exp_E(K_{\infty})} \oplus \bigoplus_v \frac{E(K_v)}{E(\fg \cO_{K, v})} \ar[r]
	& H_{\fg}(E/\cO_K) \ar[r]
	& 0.
}
\]
The map $(*)$ is clearly surjective and its kernel is $\bigoplus_v \frac{E(\fg \cO_{K, v})}{E(\ff \cO_{K, v})}$.
Then this diagram yields the claimed exact sequence.
\end{proof}

\begin{prop}\label{prop:UH_rank2}
For a modulus $\ff$ of $K$, the following hold.
\begin{itemize}
\item[(1)]
The $A$-module $H_{\ff}(E/\cO_K)$ is finite.
\item[(2)]
The $A$-module $U_{\ff}(E/\cO_K)$ is finitely generated of rank $[K: Q] - r_E(K)$.
\end{itemize}
\end{prop}

\begin{proof}
The case $\ff = (1)$ is the same as Proposition \ref{prop:UH_rank}.
Then the general case follows from Proposition \ref{prop:compar1} applied to $\fg = (1)$.
\end{proof}

%%%%%%%%%%%%%%%%%%%%%%
\subsection{$S$-modifications}\label{ss:clgp_lim}
%%%%%%%%%%%%%%%%%%%%%%

Let $\hat{A}$ be the profinite completion of $A$.
We have a canonical isomorphism $\hat{A} \simeq \prod_{\fp} A_{\fp}$, where $\fp$ runs over the primes of $A$.
For each modulus $\ff$ of $K$, we write $\hat{U}_{\ff}(E/\cO_K) = \hat{A} \otimes_A U_{\ff}(E/\cO_K)$ for the profinite completion of $U_{\ff}(E/\cO_K)$.

\begin{defn}\label{defn:clgp_moduli2}
Let $S$ be a finite set of finite places of $K$.
We set
\[
\hat{U}_S(E/\cO_K) = \varprojlim_{\ff} \hat{U}_{\ff}(E/\cO_K),
\qquad
H_S(E/\cO_K) = \varprojlim_{\ff} H_{\ff}(E/\cO_K),
\]
where $\ff$ runs over the moduli of $K$ whose supports are contained in $S$.
Here, both $\hat{U}_{\ff}(E/\cO_K)$ and $H_{\ff}(E/\cO_K)$ are projective systems with respect to $\ff$ by Proposition \ref{prop:compar1}.
\end{defn}

We have
\[
\hat{U}_{\emptyset}(E/\cO_K) = \hat{A} \otimes_A U(E/\cO_K),
\quad
H_{\emptyset}(E/\cO_K) = H(E/\cO_K).
\]
For each finite place $v$ of $K$, we see that $E(\cO_{K, v})$ is finitely generated over $\hat{A}$ by Proposition \ref{prop:loc_str}.

\begin{prop}\label{prop:compar2}
For finite sets $S, T$ of finite places of $K$ with $T \subset S$, we have an exact sequence
\[
0 \to \hat{U}_{S}(E/\cO_K) 
\to \hat{U}_{T}(E/\cO_K)
\to \bigoplus_{v \in S \setminus T} E(\cO_{K, v})
\to H_{S}(E/\cO_K)
\to H_{T}(E/\cO_K)
\to 0
\]
of $\hat{A}$-modules.
\end{prop}

\begin{proof}
First we apply the exact functor $\hat{A} \otimes_A (-)$ to the sequence in Proposition \ref{prop:compar1}.
The resulting sequence consists of finitely generated $\hat{A}$-modules, so they are all compact.
Therefore, taking the projective limit also preserves exactness, and we obtain the proposition.
\end{proof}

\begin{prop}\label{prop:UH_rank3}
Both $\hat{U}_S(E/\cO_K)$ and $H_S(E/\cO_K)$ are finitely generated over $\hat{A}$.
\end{prop}

\begin{proof}
When $S = \emptyset$, the claim follows from Proposition \ref{prop:UH_rank}; $H_{\emptyset}(E/\cO_K)$ is moreover finite.
Then the general case follows from Proposition \ref{prop:compar2} applied to $T = \emptyset$.
\end{proof}

Note that $H_S(E/\cO_K)$ can be infinite if $S$ is large; see Proposition \ref{prop:LC_rank} below.

%%%%%%%%%%%%%%%%%%%%%%
\subsection{Functoriality with respect to field extensions}\label{ss:funct}
%%%%%%%%%%%%%%%%%%%%%%

Let $K'/K$ be a finite extension.
We regard $K'$ as a global $Q$-field in the obvious way.
We also regard the given Drinfeld $A$-module $E$ over $\cO_K$ as one over $\cO_{K'}$. 

For each place $v$ of $K$, we write $K'_v = K_v \otimes_K K'$, which is isomorphic to the product of the completions of $K'$ at places lying above $v$.
When $v$ is a finite place of $K$, we also have $\fm_{K', v} \subset \cO_{K', v} \subset K'_v$ and $\kappa_v(K') = \cO_{K', v}/\fm_{K', v}$ in a similar way ($\kappa_v(K')$ is a product of finite fields).

We have the inclusion map $\iota_{K'/K}: K \hookrightarrow K'$ and trace map $\Tr_{K'/K}: K' \to K$.
By base change, these maps also induce $\iota_{K'/K}: K_v \to K'_v$ and $\Tr_{K'/K}: K'_v \to K_v$ for any place $v$ of $K$.
When $K'/K$ is inseparable, the trace map is the zero map, which is useless.
When $K'/K$ is separable, the trace map is known to be surjective and, as a result, for a nonzero ideal $\ff'$ of $\cO_{K'}$, the image $\Tr_{K'/K}(\ff')$ is again a nonzero ideal of $\cO_K$.

In this subsection, we observe that these maps $\iota_{K'/K}$ and $\Tr_{K'/K}$ induce maps between the unit groups and class groups.
We first deal with the setting of \S \ref{ss:clgp_ram}.

\begin{defn}\label{defn:fld_ext}
Let $\ff$ and $\ff'$ be moduli of $K$ and $K'$, respectively.

\begin{itemize}
\item[(1)]
When $\ff \cO_{K'} \subset \ff'$, we define natural maps
\[
U_{\ff}(E/\cO_K) \hookrightarrow U_{\ff'}(E/\cO_{K'}),
\quad
H_{\ff}(E/\cO_K) \to H_{\ff'}(E/\cO_{K'})
\]
by the commutative diagram
\[
\xymatrix{
	0 \ar[r]
	& U_{\ff'}(E/\cO_{K'}) \ar[r]
	& E(K') \ar[r]
	& \frac{E(K'_{\infty})}{\exp_E(K'_{\infty})} \oplus \bigoplus_v \frac{E(K'_v)}{E(\ff' \cO_{K', v})} \ar[r]
	& H_{\ff'}(E/\cO_{K'}) \ar[r]
	& 0\\
	0 \ar[r]
	& U_{\ff}(E/\cO_K) \ar[r] \ar@{^(->}[u]
	& E(K) \ar@{^(->}[u]_{\iota_{K'/K}} \ar[r]
	& \frac{E(K_{\infty})}{\exp_E(K_{\infty})} \oplus \bigoplus_v \frac{E(K_v)}{E(\ff \cO_{K, v})} \ar[r] \ar[u]_{\iota_{K'/K}}
	& H_{\ff}(E/\cO_K) \ar[r] \ar[u]
	& 0.
}
\]

\item[(2)]
Suppose $K'/K$ is separable.
When $\ff \supset \Tr_{K'/K}(\ff')$, we define natural maps
\[
U_{\ff'}(E/\cO_{K'}) \to U_{\ff}(E/\cO_K),
\quad
H_{\ff'}(E/\cO_{K'}) \twoheadrightarrow H_{\ff}(E/\cO_K)
\]
by the commutative diagram
\[
\xymatrix{
	0 \ar[r]
	& U_{\ff'}(E/\cO_{K'}) \ar[r] \ar[d]
	& E(K') \ar[r] \ar@{->>}[d]_{\Tr_{K'/K}}
	& \frac{E(K'_{\infty})}{\exp_E(K'_{\infty})} \oplus \bigoplus_v \frac{E(K'_v)}{E(\ff' \cO_{K', v})} \ar[r] \ar@{->>}[d]_{\Tr_{K'/K}}
	& H_{\ff'}(E/\cO_{K'}) \ar[r] \ar@{->>}[d]
	& 0\\
	0 \ar[r]
	& U_{\ff}(E/\cO_K) \ar[r]
	& E(K) \ar[r]
	& \frac{E(K_{\infty})}{\exp_E(K_{\infty})} \oplus \bigoplus_v \frac{E(K_v)}{E(\ff \cO_{K, v})} \ar[r]
	& H_{\ff}(E/\cO_K) \ar[r]
	& 0.
}
\]
\end{itemize}
\end{defn}

We now show a Galois descent property.

\begin{prop}\label{prop:descent1}
Let $K'/K$ be a finite Galois extension and $G$ its Galois group.
Let $\ff'$ be a modulus of $K'$ that is $G$-stable.
Set $\ff = \Tr_{K'/K}(\ff')$, which we regard as a modulus of $K$.
Then $H_{\ff'}(E/\cO_{K'})$ is naturally a $G$-module, and its $G$-coinvariant satisfies an isomorphism
\[
H_{\ff'}(E/\cO_{K'})_G \simeq H_{\ff}(E/\cO_K)
\]
as $A$-modules, induced by the map in Definition \ref{defn:fld_ext}(2).
\end{prop}

\begin{proof}
By the existence of normal bases, the module $K'$ is $G$-cohomologically trivial ($G$-c.t.).
Therefore, for any place $v$ of $K$, the localization $K'_v = K_v \otimes_K K'$ is also $G$-c.t.
This implies that the maps labeled (1), (2), (5), and (7) below are all isomorphic.

For each infinite place $v$ of $K$, we have a commutative diagram
\[
\xymatrix{
	 (K'_{v})_G \ar[r]^-{\exp_E} \ar[d]_{(1)}
	& E(K'_{v})_G \ar[r] \ar[d]_{(2)}
	& \left(\cfrac{E(K'_{v})}{\exp_E(K'_{v})} \right)_G \ar[r] \ar[d]_{(3)}
	& 0\\
	K_{v} \ar[r]_-{\exp_E}
	& E(K_{v}) \ar[r]
	& \cfrac{E(K_{v})}{\exp_E(K_{v})} \ar[r]
	& 0
}
\]
induced by the trace maps.
Therefore, the map (3) is isomorphic.

For each finite place $v$ of $K$, we have a commutative diagram
\[
\xymatrix{
	 E(\ff' \cO_{K', v})_G \ar[r] \ar[d]_{(4)}
	& E(K'_v)_G \ar[r] \ar[d]_{(5)}
	& \left(\cfrac{E(K'_v)}{E(\ff' \cO_{K', v})} \right)_G \ar[r] \ar[d]_{(6)}
	& 0\\
	E(\ff \cO_{K, v}) \ar[r]
	& E(K_v) \ar[r]
	& \cfrac{E(K_v)}{E(\ff \cO_{K, v})} \ar[r]
	& 0
}
\]
induced by the trace maps.
The map (4) is surjective by the definition of $\ff$.
Therefore, the map (6) is isomorphic.

We now consider the following diagram
\[
\xymatrix{
	E(K')_G \ar[r] \ar[d]_{(7)}
	& \left( \bigoplus_{v \mid \infty} \cfrac{E(K'_{v})}{\exp_E(K'_{v})} \oplus \bigoplus_{v \nmid \infty} \cfrac{E(K'_v)}{E(\ff' \cO_{K', v})} \right)_G \ar[r] \ar[d]_{(8)}
	& H_{\ff'}(E/\cO_{K'})_G \ar[r] \ar[d]_{(9)}
	& 0 \\
	E(K)  \ar[r]
	& \bigoplus_{v \mid \infty} \cfrac{E(K_{v})}{\exp_E(K_{v})} \oplus \bigoplus_{v \nmid \infty} \cfrac{E(K_v)}{E(\ff \cO_{K, v})} \ar[r]
	& H_{\ff}(E/\cO_K) \ar[r]
	& 0
}
\]
induced by the trace maps.
The map (8) is the combination of (3) and (6), so it is isomorphic.
As a consequence, (9) is isomorphic, as claimed.
\end{proof}

The following well-known proposition (Noether's theorem) is useful to handle $\Tr_{K'/K}(\ff')$.

\begin{prop}\label{prop:Tr_surj}
For each finite place $v$ of $K$, 
the trace map $\Tr_{K'/K}: K'_v \to K_v$ satisfies
\[
\Tr_{K'/K}(\cO_{K'_v}) = \cO_{K_v}
\]
if and only if $K'/K$ is tamely ramified at $v$.
\end{prop}

\begin{proof}
See, e.g., \cite[Chap.~III, Propositions 7 and  13]{Ser79}.
\end{proof}

We proceed to the setting of \S \ref{ss:clgp_lim}.

\begin{defn}\label{defn:fld_ext2}
Let $S$ and $S'$ be finite sets of finite places of $K$ and $K'$, respectively.
We write $S_{K'}$ for the set of places of $K'$ lying above places in $S$.

\begin{itemize}
\item[(1)]
When $S_{K'} \subset S'$, we define natural maps
\[
\hat{U}_S(E/\cO_K) \hookrightarrow \hat{U}_{S'}(E/\cO_{K'}),
\quad
H_S(E/\cO_K) \to H_{S'}(E/\cO_{K'})
\]
as the projective limits of the maps in Definition \ref{defn:fld_ext}(1). 

\item[(2)]
Suppose $K'/K$ is separable.
When $S_{K'} \supset S'$, we define natural maps
\[
\hat{U}_{S'}(E/\cO_{K'}) \to \hat{U}_{S}(E/\cO_K),
\quad
H_{S'}(E/\cO_{K'}) \twoheadrightarrow H_{S}(E/\cO_K)
\]
as the projective limits of the maps in Definition \ref{defn:fld_ext}(2).
\end{itemize}
\end{defn}

To ease the notation, for a finite set $S$ of finite places of $K$, we set
\[
\hat{U}_{S}(E/\cO_{K'}) = \hat{U}_{S_{K'}}(E/\cO_{K'}),
\quad
H_{S}(E/\cO_{K'}) = H_{S_{K'}}(E/\cO_{K'}).
\]
The following Galois descent property plays the key role in the proof of the main theorems.

\begin{prop}\label{prop:descent2}
Let $K'/K$ be a finite Galois extension and $G$ its Galois group.
Let $S$ be a finite set of finite places of $K$.
Suppose that $K'/K$ is tamely ramified at any finite place $v \not \in S$.
Then we have an isomorphism
\[
H_S(E/\cO_{K'})_G \simeq H_S(E/\cO_K)
\]
as $\hat{A}$-modules, induced by the map in Definition \ref{defn:fld_ext2}(2).
\end{prop}

\begin{proof}
By Proposition \ref{prop:Tr_surj} and the assumption, for any modulus $\ff'$ of $K'$ whose support is in $S_{K'}$, its trace $\Tr_{K'/K}(\ff')$ is a modulus of $K$ whose support is in $S$.
Moreover, for any modulus $\ff$ of $K$, there is a modulus $\ff'$ of $K'$ that is $G$-stable and $\ff \supset \Tr_{K'/K}(\ff')$; for instance, $\ff' = \ff \cO_{K'}$ works.
Therefore, the proposition follows by taking the projective limit of the isomorphism in Proposition \ref{prop:descent1}.
\end{proof}

%%%%%%%%%%%%%%%%%%%%%%
\section{Iwasawa-type asymptotic formula}\label{sec:app}
%%%%%%%%%%%%%%%%%%%%%%

Let $\cK/K$ be a $\Z_p$-extension of global $Q$-fields and $K_n$ its $n$-th layer.
Set $\Gamma = \Gal(\cK/K)$.
Let $S_{\ram} = S_{\ram}(\cK/K)$ denote the set of finite places $v$ that are ramified (necessarily wildly) in $\cK/K$.
We assume that $S_{\ram}(\cK/K)$ is finite.
Let $E$ be a Drinfeld $A$-module over $\cO_K$.

We first mention that there are many $\Z_p$-extensions $\cK/K$ that fit into this setting (see also Higgins \cite[Example 4.1.1]{Hig21}):

\begin{eg}\label{eg:Z_p-ext}
As in Example \ref{eg:Carlitz}, let $Q=\F_q(t)$, $\infty=1/t$, and $A=\F_q[t]$.
Assume that $K=\bF_q(\theta)$ and $\gamma\colon A\to K$ is determined by $\gamma(t)=\theta$.
There are various kinds of $\Z_p$-extensions of $K$ with finite $S_{\ram}$ as follows.

By the Kronecker--Weber-type theorem \cite[Theorem 7.1]{Hay74}, the maximal abelian extension of $K=\bF_q(\theta)$ is the composite of three linearly disjoint  fields $\ol{\F_q}(\theta)$, $K_C$, and $K_{\wtil{C}}$.
Here, $\ol{\F_q}$ is the algebraic closure of  $\F_q$ in $K^{\sep}$ and $K_C=K(C(K^{\sep})_{\tors})$, where $C(K^{\sep})_{\tors}$ is the torsion $A$-submodule of $C(K^{\sep})$.
The remained field $K_{\wtil{C}}$ is given by adjoining 
$\wtil{C}(K^{\sep})_{\tors}$ for the Drinfeld $A$-module $\wtil{C}$ determined by $\phi_{\wtil{C}}(t) = 1/\theta+\tau$.
We get $\Z_p$-extensions along these three directions:

\begin{itemize}
\item  The family $\{K_n\}_{n\geq 0}$ with $K_n=\F_{q^{p^n}}(\theta)$ gives 
a $\Z_p$-extension of $K$ unramified at all places, so that $S_{\ram}=\emptyset$.

\item For each prime $\fp$ of $A$, the $\fp$-primary part 
$C[\fp^\infty]$ of $C(K^{\sep})_{\tors}$ provides the Carlitz $\fp$-cyclotomic extension $K(C[\fp^\infty])/K$.
It is an abelian extension with Galois group $A_\fp^\times \cong (A/\fp)^\times\times\Z_p^{\aleph_0}$, in which 
all finite places except $\fp$ are unramified and $\infty$ is tamely ramified
(\cite[Proposition 12.7, Theorem 12.14]{Ros02}).
Thus, $K(C[\fp^\infty])$ contains infinitely many $\Z_p$-extensions of $K$ unramified at $\infty$ with  $S_{\ram}=\{\fp\}$.

\item The field $K_{\wtil{C}}$ contains all $\Z_p$-extensions of $K$  unramified at all finite places and totally ramified at $\infty$.
For example, for each non-constant polynomial $f(\theta)\in \F_q[\theta]$, 
there is a tower of affine smooth curves $C_n$ over $\F_q$
\[
\cdots \to C_n\to C_{n-1}\to \cdots \to C_0:=\bA^1_{\bF_q}
\]
so-called the \textit{Artin--Schreier--Witt cover} associated to $f(\theta)$ (cf.\ \cite[\S 1]{DWX16} and \cite[Example 4.10]{KW18}), 
which forms a tower of Galois covers of $\bA_{\bF_q}^1$ totally ramified at $\infty$ with total Galois group $\bZ_p$. 
Thus, writing $K_n$ for the function field of $C_n$, we get a $\Z_p$-extension of $K$ in $K_{\wtil{C}}$ with $S_{\ram}=\emptyset$. 
\end{itemize}

In general, $\Z_p$-extensions of global function fields can be described by class field theory (cf.\ \cite[\S 1]{GK88}) 
and Artin--Schreier--Witt theory (see \cite[\S 2]{KW18} for instance).
\end{eg}

%%%%%%%%%%%%%%%%%%%%%%
\subsection{Iwasawa modules}\label{ss:Iw_mod}
%%%%%%%%%%%%%%%%%%%%%%

We naturally define the Iwasawa modules.

\begin{defn}\label{defn:Iw_mod}
For a finite set $S$ of finite places of $K$, we define
\[
\hat{U}_S(E/\cO_{\cK}) = \varprojlim_n \hat{U}_S(E/\cO_{K_n}),
\quad
H_S(E/\cO_{\cK}) = \varprojlim_n H_S(E/\cO_{K_n})
\]
by using Definition \ref{defn:fld_ext2}(2).
For a prime $\fp$ of $A$, we set $H_S(E/\cO_{K_n})_{\fp} = A_{\fp} \otimes_{\hat{A}} H_S(E/\cO_{K_n})$ and define
\[
H_S(E/\cO_{\cK})_{\fp} = \varprojlim_n H_S(E/\cO_{K_n})_{\fp}
\]
similarly.
Then $H_S(E/\cO_{\cK})$ is a compact module over the completed group ring $\hat{A}[[\Gamma]] = \varprojlim_n \hat{A}[\Gamma/\Gamma^{p^n}]$, while $H_S(E/\cO_{\cK})_{\fp}$ is a compact module over $A_{\fp}[[\Gamma]] = \varprojlim_n A_{\fp}[\Gamma/\Gamma^{p^n}]$.
We write $H(E/\cO_{\cK}) = H_{\emptyset}(E/\cO_{\cK})$ and $H(E/\cO_{\cK})_{\fp} = H_{\emptyset}(E/\cO_{\cK})_{\fp}$.
\end{defn}

This $H(E/\cO_{\cK})_{\fp}$ is the same as the module introduced in \cite[\S 4.1]{Hig21}.
In fact, Propositions \ref{prop:descent4} and \ref{prop:Iw_fg} below, specialized to the case where $S = S_{\ram} = \emptyset$, are already obtained in \cite[\S 4.2]{Hig21}.

\begin{prop}\label{prop:compar3}
For finite sets $S, T$ of finite places of $K$ with $T \subset S$, we have an exact sequence
\[
0 \to \hat{U}_{S}(E/\cO_{\cK}) 
\to \hat{U}_{T}(E/\cO_{\cK})
\to \bigoplus_{v \in S \setminus T} \varprojlim_n E(\cO_{K_n, v})
\to H_{S}(E/\cO_{\cK})
\to H_{T}(E/\cO_{\cK})
\to 0
\]
of $\hat{A}[[\Gamma]]$-modules.
\end{prop}

\begin{proof}
This is the projective limit of Proposition \ref{prop:compar2}.
\end{proof}

This proposition shows that the effect of $S$ can be evaluated by $\varprojlim_n E(\cO_{K_n, v})$.
It seems to be notable that $\varprojlim_n E(\cO_{K_n, v})$ vanishes if $v$ is ramified in $\cK/K$, so the ramified places do not affect $H_{S}(E/\cO_{\cK})$ (see \S \ref{ss:ram_pl}).
This observation is unnecessary to prove the main theorem, Theorem \ref{thm:ICNF1}.

\begin{prop}\label{prop:descent4}
Let $S$ be a finite set of finite places of $K$ such that $S \supset S_{\ram}(\cK/K)$.
Then we have an isomorphism
\[
H_S(E/\cO_{\cK})_{\Gamma^{p^n}} \simeq H_S(E/\cO_{K_n})
\]
for any $n \geq 0$.
\end{prop}

\begin{proof}
This is the projective limit of Proposition \ref{prop:descent2}.
\end{proof}

\begin{prop}\label{prop:Iw_fg}
For a finite set $S$ of finite places of $K$, the module $H_S(E/\cO_{\cK})$ is finitely generated over $\hat{A}[[\Gamma]]$.
\end{prop}

\begin{proof}
By the surjective homomorphism $H_{S \cup S_{\ram}}(E/\cO_{\cK}) \twoheadrightarrow H_{S}(E/\cO_{\cK})$ in Proposition \ref{prop:compar3}, we may assume that $S \supset S_{\ram}(\cK/K)$ by enlarging $S$ if necessary.
Then by Proposition \ref{prop:descent4}, we have $H_S(E/\cO_{\cK})_{\Gamma} \simeq H_S(E/\cO_K)$, which is finitely generated over $\hat{A}$ by Proposition \ref{prop:UH_rank3}.
By applying the topological version of Nakayama's lemma (see \cite[Lemma (5.2.18)]{NSW08} or \cite[Lemma 13.16]{Was97}), the proposition follows.
\end{proof}

\begin{rem}
More generally, for any Galois extension $\cK$ of a global $Q$-field $K$, we can define
\[
H_S(E/\cO_{\cK}) = \varprojlim_{K'} H_S(E/\cO_{K'}),
\]
where $K'$ runs over the finite Galois extensions of $K$ in $\cK$.
Then the same argument shows that this is a finitely generated module over $\hat{A}[[\Gal(\cK/K)]]$ as long as $S_{\ram}(\cK/K)$ is finite and $\Gal(\cK/K)$ contains an open pro-$p$ subgroup.
It seems to be natural to apply this framework to $K(E[\fp^{\infty}])/K$, but then the Galois group is so large (see Example \ref{eg:Z_p-ext}) that the Iwasawa algebra is not even noetherian.
We do not study this topic in this paper.
\end{rem}

%%%%%%%%%%%%%%%%%%%%%%
\subsection{Iwasawa-type asymptotic formula}\label{ss:ICNF_pf}
%%%%%%%%%%%%%%%%%%%%%%

We now state and prove the main theorem, from which Theorem \ref{thm:ICNF} follows at once by setting $S = \emptyset$ (for the value of $\mu_{\fp}$, see the paragraph after the proof).
Let $H_S(E/\cO_{K_n})_{\fp, \fin}$ be the maximal finite $A_\fp$-submodule of $H_S(E/\cO_{K_n})_\fp$.

\begin{thm}\label{thm:ICNF1}
Let $\fp$ be a prime of $A$.
Let $S$ be a finite set of finite places of $K$ such that $S \supset S_{\ram}(\cK/K) \cap S_{\fp}$, that is, $S$ contains all ramified $\fp$-adic places.
Then there exist integers $\mu_{\fp} \geq 0$ and $\nu_{\fp}$ such that
\[
\length_{A_{\fp}}(H_S(E/\cO_{K_n})_{\fp, \fin}) 
= \mu_{\fp} p^n + \nu_{\fp}
\]
holds for $n \gg 0$.
Indeed, we have $\mu_{\fp} = \mu^*(H_S(E/\cO_{\cK})_{\fp, \tors})$, using Definition \ref{defn:str_thm}.
\end{thm}

\begin{proof}
When $S \supset S_{\ram} = S_{\ram}(\cK/K)$, the claim follows at once from Proposition \ref{prop:descent4} and Theorem \ref{thm:alg}(3) applied to $M = H_S(E/\cO_{\cK})_{\fp}$.
Let us deduce the general case by comparing $H_S(E/\cO_{K_n})$ with $H_{S \cup S_{\ram}}(E/\cO_{K_n})$.

For each $n$, we have an exact sequence
\[
\bigoplus_{v \in S_{\ram} \setminus S} A_{\fp} \otimes_{\hat{A}} E(\cO_{K_n, v})
\to H_{S \cup S_{\ram}}(E/\cO_{K_n})_{\fp}
\to H_{S}(E/\cO_{K_n})_{\fp}
\to 0
\]
by Proposition \ref{prop:compar2}.
By the assumption, for any $v \in S_{\ram} \setminus S$, we have $v \not\in S_{\fp}$.
Then Proposition \ref{prop:loc_str} implies
\[
A_{\fp} \otimes_{\hat{A}} E(\cO_{K_n, v}) \simeq A_{\fp} \otimes_{\hat{A}} E(\kappa_v(K_n)).
\]
The module $E(\kappa_v(K_n))$ is clearly finite.
Moreover, as $v$ is ramified in the $\Z_p$-extension $\cK/K$, the direct system $\{\kappa_v(K_n)\}_n$ is stationary.

As in Definition \ref{defn:fld_ext2}(1), the exact sequence above becomes an exact sequence of direct systems.
The first term being stationary, we obtain
\[
\length_{A_{\fp}}(H_{S \cup S_{\ram}}(E/\cO_{K_n})_{\fp, \fin})
= \length_{A_{\fp}}(H_S(E/\cO_{K_n})_{\fp, \fin}) + \constant
\]
for $n \gg 0$.
We already know that the left hand side is $\mu_{\fp} p^n + \constant$ with $\mu_{\fp} = \mu^*(H_{S \cup S_{\ram}}(E/\cO_{\cK})_{\fp, \tors})$.

It remains only to show $\mu^*(H_{S \cup S_{\ram}}(E/\cO_{\cK})_{\fp, \tors}) = \mu^*(H_{S}(E/\cO_{\cK})_{\fp, \tors})$.
We also regard the exact sequence above as an exact sequence of inverse systems as in Definition \ref{defn:fld_ext2}(2).
We have
\[
\varprojlim_n (A_{\fp} \otimes_{\hat{A}} E(\cO_{K_n, v}))
\simeq  \varprojlim_n (A_{\fp} \otimes_{\hat{A}} E(\kappa_v(K_n)))
= 0
\]
for $v \in S_{\ram} \setminus S$, which implies an isomorphism
\[
H_{S \cup S_{\ram}}(E/\cO_{\cK})_{\fp} \simeq H_S(E/\cO_{\cK})_{\fp}.
\]
This completes the proof.
\end{proof}

To establish the formula for $\mu_{\fp}$ in Theorem \ref{thm:ICNF}, 
we note that $\mu^*(H_S(E/\cO_{\cK})_{\fp, \tors}) = \mu(H_S(E/\cO_{\cK})_{\fp})$ holds if $H_S(E/\cO_K)$ is finite in Theorem \ref{thm:ICNF1}.
Indeed, by the proof we have
\[
(H_S(E/\cO_{\cK})_{\fp})_{\Gamma}
\simeq (H_{S \cup S_{\ram}}(E/\cO_{\cK})_{\fp})_{\Gamma}
\simeq H_{S \cup S_{\ram}}(E/\cO_K)_{\fp}
\]
and this is finite.
It follows that $H_S(E/\cO_{\cK})_{\fp}$ is a torsion module whose characteristic ideal is prime to $\gamma - 1$, where $\gamma$ is a topological generator of $\Gamma$.
Hence the claim follows.

Finally, we establish Corollary \ref{cor:ICNF_ur}.

\begin{proof}[Proof of Corollary \ref{cor:ICNF_ur}]
Suppose that $\cK/K$ is unramified at all finite places.
We can apply Theorem \ref{thm:ICNF} for every prime $\fp$ of $A$.
For a prime $\fp$ such that $H(E/\cO_K)_{\fp} = 0$, by Proposition \ref{prop:descent4} and Nakayama's lemma, we see that $H(E/\cO_{K_n})_{\fp} = 0$ holds for all $n \geq 0$.
Therefore, by setting $\mu = \sum_{\fp} \mu_{\fp}$ and $\nu = \sum_{\fp} \nu_{\fp}$, we obtain Corollary \ref{cor:ICNF_ur}.
\end{proof}

%%%%%%%%%%%%%%%%%%%%%%
\section{Notes on the $S$-modification}\label{sec:rank}
%%%%%%%%%%%%%%%%%%%%%%

Let $K$ be a global $Q$-field and $E$ a Drinfeld $A$-module over $\cO_K$.

%%%%%%%%%%%%%%%%%%%%%%
\subsection{An analogue of Leopoldt's conjecture}\label{ss:HU_rank}
%%%%%%%%%%%%%%%%%%%%%%

Let $\fp$ be a prime of $A$.
We naturally have the following analogue of Leopoldt's conjecture (cf.~\S \ref{ss:Leop_0}).

\begin{conj}\label{conj:LC}
The natural homomorphism
\[
A_{\fp} \otimes_A U(E/\cO_K) \to \bigoplus_{v \in S_{\fp}} A_{\fp} \otimes_{\hat{A}} E(\cO_{K, v})
\]
is injective.
\end{conj}

By Angl\`{e}s--Taelman \cite[Theorem 9.8]{AT15}, Conjecture \ref{conj:LC} is known to be true in the case where $E$ is the Carlitz module $C$ and $K/Q$ is the Carlitz $\fp$-cyclotomic extension \cite[\S 2.4]{AT15}.
The method is analogous to the proof of Leopoldt's conjecture for abelian number fields; we make use of the transcendency of logarithms, established by Bosser \cite[Appendix]{AT15} in this case.

\begin{prop}\label{prop:LC_rank}
For any finite set $S$ of finite places of $K$ such that $S \supset S_{\fp}$, the following are equivalent.
\begin{itemize}
\item[(i)]
Conjecture \ref{conj:LC} holds.
\item[(ii)]
We have $A_{\fp} \otimes_{\hat{A}} \hat{U}_{S}(E/\cO_K) = 0$.
\item[(iii)]
We have $\rank_{A_{\fp}}(H_S(E/\cO_K)_{\fp}) = r_E(K)$.
\end{itemize}
\end{prop}

\begin{proof}
Note that the kernel of the map in Conjecture \ref{conj:LC} is torsion-free over $A_{\fp}$, so (ii) is equivalent to the finiteness of $A_{\fp} \otimes_{\hat{A}} \hat{U}_{S}(E/\cO_K)$.
Then we may assume $S = S_{\fp}$ because of Proposition \ref{prop:compar2} and the finiteness of $A_{\fp} \otimes_{\hat{A}} E(\cO_{K, v})$ for $v \nmid \fp$.

We use the exact sequence in Proposition \ref{prop:compar2} with $T = \emptyset$, to which we moreover apply $A_{\fp} \otimes_{\hat{A}} (-)$.
Then the equivalence (i) $\Leftrightarrow$ (ii) is clear.
To show the equivalence (ii) $\Leftrightarrow$ (iii), we observe the ranks of the other modules:
\begin{itemize}
\item
We have $\rank_{A_{\fp}}(A_{\fp} \otimes_A U(E/\cO_K)) = [K: Q] - r_E(K)$ (Proposition \ref{prop:UH_rank}(2)).
\item
We have $\rank_{A_{\fp}}(\bigoplus_{v \mid \fp} A_{\fp} \otimes_{\hat{A}} E(\cO_{K, v})) = [K: Q]$ (Proposition \ref{prop:loc_str}).
\item
We have $\rank_{A_{\fp}}(A_{\fp} \otimes_A H(E/\cO_K)) = 0$ (Proposition \ref{prop:UH_rank}(1)).
\end{itemize}
Now (ii) $\Leftrightarrow$ (iii) follows.
\end{proof}

Let us determine the ranks of the Iwasawa modules, assuming the analogue of Leopoldt's conjecture.

\begin{prop}\label{prop:rank_Iw}
Let $\cK/K$ be a $\Z_p$-extension of global $Q$-fields such that $S_{\ram}(\cK/K)$ is finite and $\cK/K$ is totally split at any infinite places.
Suppose that Conjecture \ref{conj:LC} holds for $K_n$ for every $n$.
Then, for any finite set $S$ of finite places of $K$ such that $S \supset S_{\fp}$, we have
\[
\rank_{A_{\fp}[[\Gamma]]} H_S(E/\cO_{\cK})_{\fp} = r_E(K).
\]
\end{prop}

\begin{proof}
In the proof of Theorem \ref{thm:ICNF1}, we showed an isomorphism $H_{S \cup S_{\ram}}(E/\cO_{\cK})_{\fp}
\simeq H_S(E/\cO_{\cK})_{\fp}$.
Therefore, we may assume $S \supset S_{\ram}$.
In this case, Propositions \ref{prop:descent4} and \ref{prop:LC_rank} show
\[
\rank_{A_{\fp}} (H_S(E/\cO_{\cK})_{\fp})_{\Gamma^{p^n}}
= \rank_{A_{\fp}} H_S(E/\cO_{K_n})_{\fp} 
= r_E(K_n) = p^n r_E(K)
\]
for each $n \geq 0$, where the final equation follows from the assumption that $K_n/K$ is totally split at any infinite places.
Now Theorem \ref{thm:alg}(1) completes the proof.
\end{proof}

%%%%%%%%%%%%%%%%%%%%%%
\subsection{The effect of ramified places}\label{ss:ram_pl}
%%%%%%%%%%%%%%%%%%%%%%

Let $\cK/K$ be a $\Z_p$-extension of global $Q$-fields and $K_n$ its $n$-th layer.
We still assume that $S_{\ram} = S_{\ram}(\cK/K)$ is finite.
In the proof of Theorem \ref{thm:ICNF1}, we observed $H_{S \cup S_{\ram}}(E/\cO_{\cK})_{\fp}
\simeq H_S(E/\cO_{\cK})_{\fp}$ as long as $S \supset S_{\fp} \cap S_{\ram}$.
In this subsection, we remove the last condition:

\begin{thm}\label{thm:univ_Tr2}
For any finite set $S$ of finite places of $K$, we have $H_{S \cup S_{\ram}}(E/\cO_{\cK}) \simeq H_S(E/\cO_{\cK})$.
\end{thm}

\begin{rem}\label{rem:non-tors}
In particular, we have $H(E/\cK) \simeq H_{S_{\ram}}(E/\cK)$.
This together with Proposition \ref{prop:descent4} tells us that, although $H(E/\cK)$ is defined as $\varprojlim_n H(E/K_n)$, its Galois coinvariant describes $H_{S_{\ram}}(E/K_n)$ instead.
For this reason, it seems difficult to remove the assumption $S \supset S_{\fp} \cap S_{\ram}$ from Theorem \ref{thm:ICNF1}.
We also see that $H(E/\cK)$ is not a torsion module in general, because of Proposition \ref{prop:rank_Iw}.

This phenomenon is in contrast to number field setting.
For a $\Z_p$-extension $\cK/K$ of number fields with $K_n$ its $n$-th layer, the Iwasawa module $X(\cK) = \varprojlim_n \Z_p \otimes_{\Z} \Cl(K_n)$ is known to be torsion over the Iwasawa algebra and have enough information to describe the behavior of $\length_{\Z_p}(\Z_p \otimes \Cl(K_n))$.
\end{rem}

To prove Theorem \ref{thm:univ_Tr2}, by Proposition \ref{prop:compar3}, we only have to show $\varprojlim_n E(\cO_{K_n, v}) = 0$ for any $v \in S_{\ram}$.
This statement is independent from the Drinfeld module $E$, and follows from the following general proposition on local fields.

\begin{prop}\label{prop:univ_Tr1}
Let $F$ be a local field (i.e., a complete discrete valuation field with finite residue field) with residue characteristic $p$.
(Note that the characteristic of $F$ is either $0$ or $p$.)
Let $\cF/F$ be a ramified $\Z_p$-extension with $F_n$ its $n$-th layer.
Then we have $\varprojlim_n \cO_{F_n} = 0$, the projective limit being taken with respect to the trace maps.
\end{prop}

\begin{proof}
We may assume that the $\Z_p$-extension is totally ramified, and we only have to show $\bigcap_{n \geq 0} \Tr_{F_n/F}(\cO_{F_n}) = 0$.
Let $v_F$ and $v_{F_n}$ be the normalized valuation on $F$ and $F_n$, respectively.
Let $\cD_n = \cD_{F_n/F}$ be the different ideal of $F_n/F$ (\cite[Chap.III~\S 3]{Ser79}).
By definition $\cD_n$ is the minimal ideal of $\cO_{F_n}$ satisfying $\Tr_{F_n/F}(\cD_{F_n/F}^{-1}) \subset \cO_F$.
We also have
\[
v_F(\Tr_{F_n/F}(\cO_{F_n})) = \left \lfloor v_{F_n}(\cD_n)/p^n \right \rfloor,
\]
where we use the assumption that $F_n/F$ is totally ramified, so the ramification index is $p^n$.
Therefore, the claim is equivalent to $\lim_{n \to \infty} v_{F_n}(\cD_n)/p^n = + \infty$.

To study $\cD_n$, we use the higher ramification groups (\cite[Chap.~IV, \S 1]{Ser79}).
For $i \geq -1$, we define $G_i^{(n)} \subset \Gal(F_n/F)$ as the stabilizer subgroup of the action $\Gal(F_n/F)$ on $\cO_{F_n}/\fm_{F_n}^{i+1}$.
Then it is known (\cite[Chap.~IV, \S 1, Proposition 4]{Ser79}) that
\[
v_{F_n}(\cD_n) = \sum_{i=0}^{\infty} (\# G_i^{(n)} - 1).
\]
By the theorem of Hasse--Arf  (\cite[page 76, Example]{Ser79}), there are positive integers $i_0^{(n)}, i_1^{(n)}, \dots, i_{n-1}^{(n)}$ such that
\[
G_i^{(n)} = G \quad \text{ if $0 \leq i \leq i^{(n)}_0$},
\]
\[
G_i^{(n)} = G^{p^j} \quad \text{if $i^{(n)}_0 + pi^{(n)}_1 + \dots + p^{j-1} i^{(n)}_{j-1} < i \leq i^{(n)}_0 + pi^{(n)}_1 + \dots + p^j i^{(n)}_j$}
\]
for $j = 1, 2, \dots, n-1$, and
\[
G_i^{(n)} = \{\id\} \quad \text{if $i^{(n)}_0 + pi^{(n)}_1 + \dots + p^{n-1} i^{(n)}_{n-1} < i$}.
\]
When $n$ varies, we moreover have (\cite[page 64, Corollary]{Ser79})
\begin{align*}
i_0 & := i_0^{(1)} = i_0^{(2)} = i_0^{(3)} = \cdots,\\
i_1 & := i_1^{(2)} = i_1^{(3)} = i_1^{(4)} = \cdots,\\
i_2 & := i_2^{(3)} = i_2^{(4)} = i_2^{(5)} = \cdots,\\
\vdots
\end{align*}
Therefore, we obtain
\[
v_{F_n}(\cD_n) 
= (p^n-1) (i_0+1) + (p^{n-1} - 1) p i_1 + (p^{n-2}-1) p^2 i_2 + \dots + (p-1) p^{n-1} i_{n-1},
\]
so
\begin{align*}
v_{F_n}(\cD_n) / p^n
& = (1 - p^{-n}) (1 + i_0) + (1 - p^{-n+1}) i_1 + (1 - p^{-n+2}) i_2 + \dots + (1 - p^{-1})  i_{n-1}\\
& \geq (1 - p^{-1}) (1 + i_0 + i_1 + i_2 + \dots + i_{n-1}).
\end{align*}
Since $i_0, i_1, i_2, \dots$ are positive integers, we conclude $\lim_{n \to \infty} v_{F_n}(\cD_n)/p^n = + \infty$ as claimed.
\end{proof}

\appendix

%%%%%%%%%%%%%%%%%%%%%%
\section{The algebraic theorem}\label{sec:alg}
%%%%%%%%%%%%%%%%%%%%%%

The purpose of this section is to establish a module-theoretic version of Iwasawa's asymptotic formula in positive characteristic (Theorem \ref{thm:alg}).
In the number field setting, we study modules over the power series ring $\Z_p[[T]]$ and their quotients modulo $(1+T)^{p^n}-1$.
In the function field setting, we are naturally led to study modules over $A_{\fp}[[T]]$ and their quotients modulo $(1+T)^{p^n}-1 = T^{p^n}$.
To the best of the authors' knowledge, such an alternative has not appeared in the literature.

From now on, let $R$ be a complete discrete valuation ring whose characteristic is a prime number $p$, which plays the role of $A_{\fp}$ in the main text.
We write $\ord_R: R \to \bN \cup \{\infty \}$ for the normalized valuation on $R$.
We write $(-)_{\fin}$ for the maximal submodule of finite length for finitely generated $R$-modules.

Let $\Gamma$ be a topological group that is isomorphic to $\Z_p$.
Let
\[
R[[\Gamma]] := \varprojlim_n R[\Gamma/\Gamma^{p^n}]
\]
be the completed group ring.
We have the following analogue of Serre's isomorphism (see \cite[Proposition (5.3.5)]{NSW08} or \cite[Theorem 7.1]{Was97}).

\begin{lem}\label{lem:gprg}
Let $\gamma$ be a topological generator of $\Gamma$.
We have isomorphisms
\[
R[\Gamma/\Gamma^{p^n}] 
\simeq R[T]/(T^{p^n})
\]
for $n \geq 0$ and
\[
R[[\Gamma]] \simeq R[[T]]
\]
by sending $\gamma$ to $1 + T$.
\end{lem}

\begin{proof}
The first one follows from
\[
R[\Gamma/\Gamma^{p^n}] 
\simeq R[T]/((1+T)^{p^n}-1)
= R[T]/(T^{p^n})
\]
since $p = 0$ in $R$.
By taking the inverse limit, the second one follows.
\end{proof}

For each $f \in R[[T]] \setminus \{0\}$, 
we define $\ord_T(f) \in \bN$ and $f^* \in R[[T]] \setminus T R[[T]]$ by
\[
f(T) = T^{\ord_T(f)} f^*(T).
\]
Note that $f^* \not \in T R[[T]]$ is equivalent to $f^*(0) \neq 0$.

We briefly recall the structure theorem for $R[[T]]$-modules (see \cite[Chapter V, \S 1]{NSW08}).

\begin{defn}\label{defn:str_thm}
An $R[[T]]$-module is said to be pseudo-null if it is of finite length.
An $R[[T]]$-homomorphism between finitely generated modules is said to be pseudo-isomorphic if its kernel and cokernel are both pseudo-null.

For a finitely generated $R[[T]]$-module $M$, the structure theorem states that there exists a pseudo-isomorphism
\[
\varphi: M \to R[[T]]^r \oplus \bigoplus_{i = 1}^s R[[T]]/(f_i)
\]
for $r = \rank_{R[[T]]}(M)$ and some non-zero elements $f_1, \dots, f_s \in R[[T]]$.
The target module of $\varphi$ is often called an elementary module.
In this case, the $R[[T]]$-torsion part of $M$, denoted by $M_{\tors}$, is pseudo-isomorphic to $\bigoplus_{i = 1}^s R[[T]]/(f_i)$.
The characteristic ideal of $M_{\tors}$ is defined as $\cha(M_{\tors}) = (f_1 \cdots f_s)$.

Suppose $M$ is a finitely generated torsion $R[[T]]$-module.
Let $F$ be a generator of $\cha(M)$ (i.e., $\cha(M) = (F)$).
We define $\mu(M) = \ord_R(F(0))$ if $F$ is prime to $T$, and $\mu^*(M) = \ord_R(F^*(0))$ in general.
For a finitely generated torsion $R[[\Gamma]]$-module, we also define $\mu(M)$ and $\mu^*(M)$ by using the isomorphism $R[[\Gamma]] \simeq R[[T]]$ in Lemma \ref{lem:gprg}.
\end{defn}

Now we state the main theorem of this section.
Claim (2) essentially suffices to prove Theorem \ref{thm:ICNF}, while claim (3) is necessary to prove Theorem \ref{thm:ICNF1}.

\begin{thm}\label{thm:alg}
Let $M$ be a finitely generated $R[[\Gamma]]$-module.

\begin{itemize}
\item[(1)]
We have
\[
\rank_R(M_{\Gamma^{p^n}}) = \rank_{R[[\Gamma]]}(M) p^n + \constant
\]
for $n \gg 0$.
\item[(2)]
If $M_{\Gamma}$ is finite,
then we have
\[
\length_R(M_{\Gamma^{p^n}}) = \mu(M) p^n + \constant
\]
for $n \gg 0$.
\item[(3)]
In general, we have
\[
\length_R((M_{\Gamma^{p^n}})_{\fin}) = \mu^*(M_{\tors}) p^n + \constant
\]
for $n \gg 0$.
\end{itemize}
\end{thm}

For an $R[[T]]$-module $M$ and an integer $N \geq 0$, we write 
\[
M[T^N] = \{x \in M \mid T^N x = 0 \},
\quad
M/T^N = M/T^N M.
\]
We also write $M[T^\infty] = \bigcup_{N \geq 0} M[T^N]$.

By the isomorphism in Lemma \ref{lem:gprg}, we see that Theorem \ref{thm:alg} is a consequence of the following.

\begin{thm}\label{thm:alg_T}
Let $M$ be a finitely generated $R[[T]]$-module.
\begin{itemize}
\item[(1)]
We have
\[
\rank_R(M/T^N) = \rank_{R[[T]]}(M) N + \constant
\]
for $N \gg 0$.
Indeed, the constant is $\ord_T(F)$ with $\cha(M_{\tors}) = (F)$.
\item[(2)]
If $M/T$ is finite,
then we have
\[
\length_R(M/T^N) = \mu(M) N + \constant
\]
for $N \gg 0$.
\item[(3)]
In general, we have
\[
\length_R((M/T^N)_{\fin}) = \mu^*(M_{\tors}) N + \constant
\]
for $N \gg 0$.
\end{itemize}
\end{thm}

First we show claim (1):

\begin{proof}[Proof of Theorem \ref{thm:alg_T}(1)]
We take a pseudo-isomorphism $\varphi$ as in Definition \ref{defn:str_thm}.
For any $N \geq 0$, we obtain
\begin{align*}
\rank_R(M/T^N)
& = \rank_R \left((R[[T]]/T^N) ^r \oplus \bigoplus_{i = 1}^s R[[T]]/(f_i, T^N) \right)\\
& = r N + \sum_{i = 1}^s \min \{\ord_T(f_i), N\}.
\end{align*}
The final term is constant as long as $N \geq \ord_T(f_i)$ for any $i$.
In fact, for such an $N$, we have
\[
\sum_{i = 1}^s \min \{\ord_T(f_i), N\}
= \sum_{i = 1}^s \ord_T(f_i)
= \ord_T(F),
\]
so the claim holds.
\end{proof}

The rest of this section is devoted to the proof of Theorem \ref{thm:alg_T}(2)(3).

\begin{lem}\label{lem:alg0-0}
Let $f \in R[[T]] \setminus \{0\}$.
\begin{itemize}
\item[(1)]
If $f(0) \neq 0$, we have
\[
\length_R (R[[T]]/(f, T^N)) = \ord_R(f(0)) N
\]
for any $N \geq 0$
\item[(2)]
In general, we have
\[
\length_R (R[[T]]/(f, T^N))_{\fin} = \ord_R(f^*(0)) (N - \ord_T(f))
\]
for any $N \geq \ord_T(f)$.
\end{itemize}
\end{lem}

\begin{proof}
(1)
Let $f(T) = a_0 + a_1 T + a_2 T^2 + \cdots$ with $a_0, a_1, a_2, \dots \in R$.
The $R$-module $R[[T]]/(T^N)$ is free of rank $N$ with a basis $1, T, T^2, \dots, T^{N-1}$.
With respect to this basis, the $R$-endomorphism given by the multiplication by $f$ is presented by a triangular matrix 
\[
\begin{pmatrix}
	a_0 & 0 & \cdots & 0 & 0\\
	a_1 & a_0 & \cdots & 0 & 0\\
	\vdots & \vdots & \ddots & \vdots & \vdots\\
	a_{N-2} & a_{N-3} & \cdots & a_0 & 0\\
	a_{N-1} & a_{N-2} & \cdots & a_1 & a_0
\end{pmatrix}.
\]
The determinant of this matrix is $a_0^N$, so we have
\[
\length_R(R[[T]]/(f, T^N)) 
= \ord_R(a_0^N)
= \ord_R(a_0) N.
\]
Since $a_0 = f(0)$, the formula holds.

(2)
For $N \geq \ord_T(f)$, we have
\[
(R[[T]]/(f, T^N))_{\fin}
= (T^{\ord_T(f)})/(f, T^N)
\simeq R[[T]]/(f^*, T^{N - \ord_T(f)}).
\]
Now the claim follows from (1).
\end{proof}

We show Theorem \ref{thm:alg_T}(2)(3) for elementary modules:

\begin{prop}\label{prop:alg_T_elem}
Let $E$ be an $R[[T]]$-module of the form $E = R[[T]]^r \oplus \bigoplus_{i = 1}^s R[[T]]/(f_i)$.
Set $F = f_1 \cdots f_s$.
\begin{itemize}
\item[(1)]
Suppose that $r = 0$ and $F(0) \neq 0$.
Then we have 
\[
\length_R(E/T^N) = \ord_R(F(0)) N
\]
for any $N \geq 0$.
\item[(2)]
In general, we have 
\[
\length_R((E/T^N)_{\fin}) = \ord_R(F^*(0)) N + \constant
\]
for $N \gg 0$.
\end{itemize}
\end{prop}

\begin{proof}
(1)
By Lemma \ref{lem:alg0-0}(1), we have
\begin{align*}
\length_R(E/T^N)
& = \sum_{i=1}^s \length_R(R[[T]]/(f_i, T^N))\\
& = \sum_{i=1}^s \ord_R(f_i(0)) N\\
& = \ord_R(F(0)) N
\end{align*}
as claimed.

(2)
By Lemma \ref{lem:alg0-0}(2), we have
\begin{align*}
\length_R((E/T^N)_{\fin})
& = r \length_R((R[[T]]/T^N)_{\fin}) + \sum_{i=1}^s \length_R((R[[T]]/(f_i, T^N))_{\fin})\\
& = \sum_{i=1}^s \ord_R(f_i^*(0)) (N - \ord_T(f_i))\\
& = \ord_R(F^*(0)) N - \sum_{i=1}^s \ord_R(f_i^*(0)) \ord_T(f_i)
\end{align*}
for $N \gg 0$ (indeed, $N \geq \ord_T(f_i)$ for any $i$ suffices), as claimed.
\end{proof}

To complete the proof of Theorem \ref{thm:alg_T}(2)(3), we show that the pseudo-isomorphism does not affect the lengths up to constants.
This is a stronger statement than merely being bounded as $N \to \infty$, which is much easier.

\begin{lem}\label{lem:alg_L}
Let $0 \to M' \to M \to M'' \to 0$ be an exact sequence of finitely generated $R[[T]]$-modules.
Let $L$ be the cokernel of the induced map $M[T^\infty] \to M''[T^\infty]$.
\begin{itemize}
\item[(1)]
We have an exact sequence
\[
0 \to L \to M'/T^N \to M/T^N \to M''/T^N \to 0
\]
for $N \gg 0$.
\item[(2)]
The sequences in (1) fit into commutative diagrams
\[
\xymatrix{
	0 \ar[r]
	& L \ar[r] \ar[d]_{T \times}
	& M'/T^{N+1} \ar[r] \ar@{->>}[d]
	& M/T^{N+1} \ar[r] \ar@{->>}[d]
	& M''/T^{N+1} \ar[r] \ar@{->>}[d]
	& 0\\
	0 \ar[r]
	& L \ar[r]
	& M'/T^N \ar[r]
	& M/T^N \ar[r]
	& M''/T^N \ar[r]
	& 0
}
\]
for $N \gg 0$, where the right three vertical arrows are the natural surjective maps.
\item[(3)]
The sequences in (1) fit into commutative diagrams
\[
\xymatrix{
	0 \ar[r]
	& L \ar[r]
	& M'/T^{N+1} \ar[r]
	& M/T^{N+1} \ar[r]
	& M''/T^{N+1} \ar[r]
	& 0\\
	0 \ar[r]
	& L \ar[r] \ar@{=}[u]
	& M'/T^N \ar[r] \ar[u]_{T \times}
	& M/T^N \ar[r] \ar[u]_{T \times}
	& M''/T^N \ar[r] \ar[u]_{T \times}
	& 0
}
\]
for $N \gg 0$.
\end{itemize}
\end{lem}

\begin{proof}
(1)
For any $N \geq 0$, by the snake lemma we have an exact sequence
\[
0 \to M'[T^N] \to M[T^N] \to M''[T^N]
\to M'/T^N \to M/T^N \to M''/T^N \to 0.
\]
Since both $M$ and $M''$ are noetherian, we have $M[T^N] = M[T^\infty]$ and $M''[T^N] = M''[T^\infty]$ for $N \gg 0$.
This shows the exact sequence.

(2)(3)
The commutativity of the diagrams follows from the construction.
\end{proof}

\begin{prop}\label{prop:PI_const1}
Let $\varphi: M \to E$ be a pseudo-isomorphism between finitely generated torsion $R[[T]]$-modules.
Suppose that the characteristic ideals are prime to $T$.
Then 
\[
\length_R(M/T^N) = \length_R(E/T^N) + \constant
\]
for $N \gg 0$.
\end{prop}

\begin{proof}
By dividing the exact sequence $0 \to \Ker \varphi \to M \overset{\varphi}{\to} E \to \Cok \varphi \to 0$ into two short exact sequences, it is enough to show the following.
Let $0 \to M' \to M \to M'' \to 0$ be an exact sequence of finitely generated torsion $R[[T]]$-modules whose characteristic ideals are prime to $T$.
Then the following are true for $N \gg 0$:
\begin{itemize}
\item[(a)]
If $M'$ is pseudo-null, then $\length_R(M/T^N) = \length_R(M''/T^N) + \constant$.
\item[(b)]
If $M''$ is pseudo-null, then $\length_R(M/T^N) = \length_R(M'/T^N) + \constant$.
\end{itemize}
To show these, we define $L$ as in Lemma \ref{lem:alg_L}.
Then Lemma \ref{lem:alg_L}(1) implies
\[
\length_R(M/T^N) = \length_R(M'/T^N) + \length_R(M''/T^N) - \length_R(L)
\]
for $N \gg 0$.
In case (a), $\length_R(M'/T^N)$ is constant as $N \to \infty$.
In case (b), $\length_R(M''/T^N)$ is constant as $N \to \infty$.
Therefore, both claims hold.
\end{proof}

We are already able to prove claim (2) of Theorem \ref{thm:alg_T}, but claim (3) requires two more propositions.

\begin{prop}\label{prop:PI_const2}
Let $\varphi: M \to E$ be a pseudo-isomorphism between finitely generated torsion $R[[T]]$-modules.
Then 
\[
\length_R((M/T^N)_{\fin}) = \length_R((E/T^N)_{\fin}) + \constant
\]
for $N \gg 0$.
\end{prop}

\begin{proof}
As in the proof of Proposition \ref{prop:PI_const1}, it is enough to show the following.
Let $0 \to M' \to M \to M'' \to 0$ be an exact sequence of finitely generated torsion $R[[T]]$-modules.
Then the following are true for $N \gg 0$:
\begin{itemize}
\item[(a)]
If $M'$ is pseudo-null, then $\length_R((M/T^N)_{\fin}) = \length_R((M''/T^N)_{\fin}) + \constant$.
\item[(b)]
If $M''$ is pseudo-null, then $\length_R((M/T^N)_{\fin}) = \length_R((M'/T^N)_{\fin}) + \constant$.
\end{itemize}
Let us write $\pi_N: M/T^{N+1} \to M/T^N$ for the natural surjective map, so we have an exact sequence
\[
0 \to \Ker \pi_N \to M/T^{N+1} \overset{\pi_N}{\to} M/T^N \to 0.
\]
Since $M$ is torsion, the $R$-rank of $M/T^N$ stabilizes (see Theorem \ref{thm:alg_T}(1)), so $\Ker \pi_N$ is of finite length for $N \gg 0$.
Therefore, we have an exact sequence
\[
0 \to \Ker \pi_N \to (M/T^{N+1})_{\fin} \overset{\pi_N}{\to} (M/T^N)_{\fin} \to 0,
\]
which implies
\[
\length_R((M/T^{N+1})_{\fin}) - \length_R((M/T^N)_{\fin}) = \length_R(\Ker \pi_N)
\]
for $N \gg 0$.
Similar formulas hold for $M'$ and $M''$, concerning $\pi'_N: M'/T^{N+1} \to M'/T^N$ and $\pi''_N: M''/T^{N+1} \to M''/T^N$.
Therefore, claims (a) and (b) can be restated as follows (for $N \gg 0$):
\begin{itemize}
\item[(a')]
If $M'$ is pseudo-null, then $\length_R(\Ker \pi_N) = \length_R(\Ker \pi_N'')$.
\item[(b')]
If $M''$ is pseudo-null, then $\length_R(\Ker \pi_N) = \length_R(\Ker \pi_N')$.
\end{itemize}
To show these, we define $L$ as in Lemma \ref{lem:alg_L}.

(a')
Suppose $M'$ is pseudo-null.
Note that Lemma \ref{lem:alg_L}(1) implies that $L$ is of finite length.
For $N \gg 0$, the map $\pi'_N$ is isomorphic, so Lemma \ref{lem:alg_L}(2) implies
\[
L[T] = 0,
\quad
0 \to L/T \to \Ker \pi_N \to \Ker \pi''_N \to 0.
\]
Since $L$ is of finite length, $L[T] = 0$ implies $L/T = 0$.
Therefore, we obtain $\Ker \pi_N \simeq \Ker \pi''_N$, so the claim holds.

(b')
Suppose $M''$ is pseudo-null.
Note that the definition of $L$ implies that $L$ is of finite length.
For $N \gg 0$, the map $\pi_N''$ is isomorphic, so Lemma \ref{lem:alg_L}(2) implies
\[
0 \to L[T] \to \Ker \pi_N' \to \Ker \pi_N \to L/T \to 0.
\]
This shows the claim.
\end{proof}

\begin{prop}\label{prop:PI_const3}
Let $M$ be a finitely generated $R[[T]]$-module.
We set $M_{\tors}$ as the torsion part of $M$.
Then 
\[
\length_R((M_{\tors}/T^N)_{\fin}) = \length_R((M/T^N)_{\fin}) + \constant
\]
for $N \gg 0$.
\end{prop}

\begin{proof}
We first suppose that $M$ is torsion-free and claim that
\[
T \times: (M/T^N)_{\fin} \to (M/T^{N+1})_{\fin}
\]
 is isomorphic for $N \gg 0$.
By the structure theorem, we can take $0 \to M \to F \to M'' \to 0$, where $F$ is a free $R[[T]]$-module and $M''$ is pseudo-null (indeed, $F$ is the reflexive hull of $M$).
By Lemma \ref{lem:alg_L}(1), we have a module $L$, which is of finite length because so is $M''$, that fits into an exact sequence
\[
0 \to L \to M/T^N \to F/T^N \to M''/T^N \to 0
\]
for $N \gg 0$.
We have $(F/T^N)_{\fin} = 0$ as $F$ is a free module.
Therefore, we obtain $L \simeq (M/T^N)_{\fin}$, and the commutative diagram in Lemma \ref{lem:alg_L}(3) shows the claim.

Now we deal with the general case.
We set $M_{/\tors} = M/M_{\tors}$, so we have an exact sequence
\[
0 \to M_{\tors} \to M \to M_{/\tors} \to 0.
\]
We apply Lemma \ref{lem:alg_L}(3) to this sequence; since $M_{/\tors}$ is torsion-free, the module $L$ vanishes.
Then, applying the left exact functor $(-)_{\fin}$, we obtain a commutative diagram with exact rows
\[
\xymatrix{
	0 \ar[r]
	& (M_{\tors}/T^{N+1})_{\fin} \ar[r] 
	& (M/T^{N+1})_{\fin} \ar[r]
	& ((M_{/\tors})/T^{N+1})_{\fin} \\
	0 \ar[r]
	& (M_{\tors}/T^N)_{\fin} \ar[r] \ar[u]_{T \times}
	& (M/T^N)_{\fin} \ar[r] \ar[u]_{T \times}
	& ((M_{/\tors})/T^N)_{\fin}. \ar[u]_{T \times}
}
\]
By the claim above applied to the torsion-free module $M_{/\tors}$, the right vertical arrow is isomorphic when $N \gg 0$.
Therefore, the images to the right-most terms stabilize as $N \gg 0$, which shows
\[
\length_R((M/T^N)_{\fin}) - \length_R((M_{\tors}/T^N)_{\fin}) = \constant
\]
for $N \gg 0$.
This completes the proof.
\end{proof}

\begin{proof}[Proof of Theorem \ref{thm:alg_T}(2)(3)]
Claim (2) follows at once by Propositions \ref{prop:alg_T_elem}(1) and \ref{prop:PI_const1}.
Claim (3) follows at once by Propositions \ref{prop:alg_T_elem}(2), \ref{prop:PI_const2}, and \ref{prop:PI_const3}.
\end{proof}

This completes the proof of Theorem \ref{thm:alg_T}, and hence of Theorem \ref{thm:alg}.

%%%%%%%%%%%%%%%%%%%%%%
\section{Ideal class groups}\label{sec:NT}
%%%%%%%%%%%%%%%%%%%%%%

In this section, we review classical theory on ideal class groups.
The frameworks in \S \ref{ss:cl_cl}, \S \ref{ss:cl_Iw}, and \S \ref{ss:Leop_0} serve as prototypes for the arguments in \S \ref{sec:clgp}, \S \ref{sec:app}, and \S \ref{ss:HU_rank}, respectively.

\begin{table}[htbp]
\centering
\begin{tabular}{ccc}
\toprule
Number fields & & Function fields \\
\midrule
$\Z \subset \Z_p$ ($p\Z \subset \Z$) & 
\multirow{2}{*}{Coefficient rings} &  $A \subset A_{\fp}$ ($\fp \subset A$)\\
$\Z \subset \Q \subset \R$ & & $A \subset Q \subset Q_{\infty}$\\
$K$: a number field ($\Q \hookrightarrow K$) & Base fields &  $K$: a global $Q$-field ($Q \overset{\gamma}{\hookrightarrow} K$) \\
$\bG_{\textrm{m}}$: the multiplicative group &
\multirow{2}{*}{Functors} &  $E$: a Drinfeld $A$-module over $\cO_K$\\
$\text{($\cO_K$-Alg)} \to \text{($\Z$-Mod)}$ &  &  
$\text{($\cO_K$-Alg)} \to \text{($\cO_K\{\tau\}$-Mod)} \to \text{($A$-Mod)}$\\
$U(K)$, $U_{\ff}(K)$, $\hat{U}_S(K)$ & Unit groups &  $U(E/\cO_K)$, $U_{\ff}(E/\cO_K)$, $\hat{U}_S(E/\cO_{K})$\\
$\Cl(K)$, $\Cl_{\ff}(K)$, $\Cl_S(K)$ & Class groups &  $H(E/\cO_K)$, $H_{\ff}(E/\cO_K)$, $H_S(E/\cO_{K})$\\
$X_S(\cK)$ & Iwasawa modules &  $H_S(E/\cO_{\cK})_{\fp}$\\
$r_2(K)$ & (Expected) ranks & $r_E(K)$\\
\bottomrule
\end{tabular}
\caption{Analogies between number fields and function fields}
\label{table:analogies}
\end{table}

%%%%%%%%%%%%%%%%%%%%%%
\subsection{The ideal class groups and unit groups}\label{ss:cl_cl}
%%%%%%%%%%%%%%%%%%%%%%

Let $K$ be a number field, i.e., a finite extension of the rationals $\Q$.
The integer ring $\cO_K \subset K$ is defined as the integral closure of $\Z$ in $K$.
We write $U(K) = \cO_K^{\times}$ for the unit group and $\Cl(K)$ for the class group of $K$.
These fit into a natural exact sequence
\[
0 \to U(K) \to K^{\times} 
\to \bigoplus_v \frac{K_v^{\times}}{\cO_{K, v}^{\times}}
\to \Cl(K) \to 0,
\]
where $v$ runs over all finite primes of $K$ and $\cO_{K, v} \subset K_v$ denote the completions of $\cO_K \subset K$ (cf.~Definition \ref{defn:clgp_ori} and Remark \ref{rem:clgp_var}).
Fundamental theorems in algebraic number theory show that (cf.~Proposition \ref{prop:UH_rank})
\begin{itemize}
\item[(1)]
The $\Z$-module $\Cl(K)$ is finite.
\item[(2)]
The $\Z$-module $U(K)$ is finitely generated of rank $[K: \Q] - (r_2(K) + 1)$, where $r_2(K)$ denotes the number of complex places of $K$.
\end{itemize}

Let $\ff$ be a nonzero ideal of $\cO_K$, which we regard as a modulus $\prod_v v^{\ord_v(\ff)}$.
We define the unit group $U_{\ff}(K)$ and class group $\Cl_{\ff}(K)$ by the exact sequence (cf.~Definition \ref{defn:clgp_moduli})
\[
0 \to U_{\ff}(K) \to K^{\times} 
\to \bigoplus_v \frac{K_v^{\times}}{U^{(\ff)}(K_v)}
\to \Cl_{\ff}(K) \to 0,
\]
where $U^{(\ff)}(K_v) = \{x \in \cO_{K, v}^{\times} \mid x \equiv 1 \mod \ff \cO_{K, v} \}$ (note that $U^{(\ff)}(K_v) = \cO_{K, v}^{\times}$ if $\ord_v(\ff) = 0$).
Then the analogues of Propositions \ref{prop:compar1} and \ref{prop:UH_rank2} are valid.

We write $\hat{\Z}$ for the profinite completion of $\Z$, and set $\hat{U}_{\ff}(K) = \hat{\Z} \otimes_{\Z} U_{\ff}(K)$.
For a finite set $S$ of finite primes of $K$, we set (cf.~Definition \ref{defn:clgp_moduli2})
\[
\hat{U}_S(K) = \varprojlim_{\ff} \hat{U}_{\ff}(K),
\qquad
\Cl_S(K) = \varprojlim_{\ff} \Cl_{\ff}(K),
\]
where $\ff$ runs over moduli of $K$ whose supports are in $S$.
Then the analogues of Propositions \ref{prop:compar2} and \ref{prop:UH_rank3} are valid.

We also have the analogues of Definitions \ref{defn:fld_ext} and \ref{defn:fld_ext2}.
However, contrary to Propositions \ref{prop:descent1} and \ref{prop:descent2}, we do not have a nice Galois descent property for the ideal class groups.
This is because we have to deal with the multiplicative groups like $(K')^{\times}$ and $(K'_v)^{\times}$ as Galois modules, which are not necessarily cohomologically trivial.
Indeed, the cohomology groups are studied by class field theory.

%%%%%%%%%%%%%%%%%%%%%%
\subsection{Iwasawa modules}\label{ss:cl_Iw}
%%%%%%%%%%%%%%%%%%%%%%

Let $\cK/K$ be a $\Z_p$-tower of number fields and let $K_n$ be its $n$-th layer.
It is known that only $p$-adic primes can ramify in a $\Z_p$-extension, so there are only finitely many ramified primes; this property does not hold in the function field case in general.

For a finite set $S$ of finite primes of $K$, we define the Iwasawa module
\[
X_S(\cK) = \varprojlim_n (\Z_p \otimes_{\hat{\Z}} \Cl_S(K_n)),
\]
which is a compact $\Z_p[[\Gamma]]$-module.
The failure of the Galois descent property makes it difficult to study $\Z_p \otimes_{\Z} \Cl(K_n)$ by using $X(\cK) = X_{\emptyset}(\cK)$.
However, this was overcome in the original Iwasawa theory, leading to the result described in \S \ref{ss:intro_1}.
Indeed, we know that $X(\cK)$ is torsion over the Iwasawa algebra and the $\lambda$ and $\mu$ are the invariants attached to $X(\cK)$.
This contrasts to the Taelman class groups as mentioned in Remark \ref{rem:non-tors}.

%%%%%%%%%%%%%%%%%%%%%%
\subsection{Leopoldt's conjecture}\label{ss:Leop_0}
%%%%%%%%%%%%%%%%%%%%%%

See \cite[Chapter X, \S 3]{NSW08} or \cite[\S 5.5, \S 13.5]{Was97} for the details.
For a number field $K$ and a prime number $p$, Leopoldt's conjecture claims that the natural homomorphism
\[
\Z_p \otimes_{\Z} U(K) \to \bigoplus_{v \in S_p} \Z_p \otimes_{\hat{\Z}} \cO_{K, v}^{\times}
\]
is injective, where $S_p$ denotes the set of $p$-adic primes of $K$ (cf.~Conjecture \ref{conj:LC}).
Leopoldt's conjecture is known to be true if $K/\Q$ is abelian (Brumer's theorem).
For each $S \supset S_p$, the following are equivalent (cf.~Proposition \ref{prop:LC_rank}):
\begin{itemize}
\item[(i)]
Leopoldt's Conjecture holds.
\item[(ii)]
We have $\Z_p \otimes_{\hat{\Z}} \hat{U}_S(K) = 0$.
\item[(iii)]
We have $\rank_{\Z_p}(\Z_p \otimes_{\hat{\Z}} \Cl_S(K)) = r_2(K) + 1$.
\end{itemize}

For a $\Z_p$-extension $\cK/K$, Leopoldt's conjecture for the intermediate fields implies
\[
\rank_{\Z_p[[\Gamma]]}(X_S(\cK)) = r_2(K)
\]
for any $S \supset S_p$ (cf.~Proposition \ref{prop:rank_Iw}).
Indeed, for this result, a weak form of Leopoldt's conjecture suffices, which is known to be true if $\cK/K$ is the cyclotomic $\Z_p$-extension.

\section*{Acknowledgments}

The first author is supported by JSPS KAKENHI Grant Number 22K13898.

{
\bibliographystyle{amsalpha}
\bibliography{biblio}
}

\end{document}